\documentclass{article}

\usepackage{hyperref}
\usepackage[bibtex-style]{amsrefs}
\usepackage{amsmath, amsthm, amssymb, amsfonts}
\usepackage[utf8]{inputenc}
\usepackage[T1]{fontenc}
\usepackage{bm}
\usepackage{tikz}
\usepackage[capitalize]{cleveref}
\usetikzlibrary{calc}
\usetikzlibrary{decorations}
\usetikzlibrary{decorations.pathreplacing}
\usetikzlibrary{shapes}
\usetikzlibrary{arrows, automata, decorations.pathreplacing, fit, matrix, patterns, positioning}
\usepackage{tikz-qtree}
\usepackage{xifthen} 
\usepackage{xspace}
\usepackage{bbold}
\usepackage[mathscr]{euscript}
\usepackage{enumitem}
\usepackage{todonotes}
\usepackage{pgf}
\pgfmathsetseed{\number\pdfrandomseed}

\newcommand\absdot[2]{
	\node at #1 {\normalsize $\bullet$};
	\node at #1 [below] {$#2$};
}

\newcommand{\plotperm}[1]{
	\foreach \j [count=\i] in {#1} {
		\absdot{(\i,\j)}{};
	};
}

\newcommand{\plotpermgraph}[1]{
	\foreach \j [count=\i] in {#1} {
		\foreach \b [count=\a] in {#1} {
			\ifthenelse{\a<\i \AND \b>\j}{\draw (\a,\b)--(\i,\j);}{}
		};
	};
	\plotperm{#1};
}

\newtheorem{theorem}{Theorem}
\newtheorem{corollary}[theorem]{Corollary}
\newtheorem{lemma}[theorem]{Lemma}

\newtheorem{proposition}[theorem]{Proposition}
\newtheorem{observation}[theorem]{Observation}
\newtheorem{question}[theorem]{Question}

\theoremstyle{remark}
\newtheorem{definition}[theorem]{Definition}
\newtheorem{remark}[theorem]{Remark}

\DeclareMathOperator{\Grid}{Grid}

\newcommand{\C}{\ensuremath{\mathcal{C}}\xspace}
\newcommand{\D}{\ensuremath{\mathcal{D}}\xspace}
\newcommand{\I}{\ensuremath{\mathcal{I}}\xspace}
\renewcommand{\S}{\ensuremath{\mathfrak{S}}} 
\newcommand{\T}{\ensuremath{\mathcal{T}}\xspace}
\newcommand{\X}{\ensuremath{\mathcal{X}}\xspace}
\newcommand{\Av}{\mathop{\mathrm{Av}}}

\newcommand{\sortS}{\ensuremath{\mathbf{S}}\xspace}
\newcommand{\sortB}{\ensuremath{\mathbf{B}}\xspace}
\newcommand{\sortQ}{\ensuremath{\mathbf{Q}}\xspace}
\newcommand{\rank}{\mathop{rank}\xspace}
\newcommand{\card}{\mathop{card}\xspace}
\newcommand{\EFeq}{\sim}

\renewcommand{\th}{^{\mbox{\scriptsize th}}}

\newcommand{\inc}{\iota} 
\newcommand{\dec}{\delta} 

\newcommand{\inter}{\ensuremath{\mathsf{TO(TO\hspace*{-0.2em}+\hspace*{-0.2em}OB)}}\xspace}
\newcommand{\TOTO}{\ensuremath{\mathsf{TOTO}}\xspace} 
\newcommand{\TOOB}{\ensuremath{\mathsf{TOOB}}\xspace} 
\newcommand{\TOTOC}{\ensuremath{\mathsf{TOTO}(\mathcal{C})}\xspace} 

\renewcommand{\le}{\leqslant}
\renewcommand{\leq}{\leqslant}

\def\si{\sigma}
\def\la{\lambda}
\def\La{\Lambda}
\def\CCC{\mathcal{C}}
\def\SSS{\mathcal{S}} 
\def\MMM{\mathcal{M}} 
\def\DDD{\mathcal{D}}
\def\BA{\bm{\mathcal{B}}}
\DeclareMathOperator{\ct}{ct}
\DeclareMathOperator{\grow}{gr}
\DeclareMathOperator{\qdepth}{qd}

\def\et{\, \wedge \,}
\def\ou{\, \vee \,}

\newcommand\revision[1]{{#1}}

\title{Two first-order logics of permutations}
\author{Michael Albert, Mathilde Bouvel, Valentin F\'eray}
\date{}
\begin{document}
\maketitle

\begin{abstract}
  We consider two orthogonal points of view on finite permutations,
  seen as pairs of linear orders (corresponding to the usual one line representation of permutations as words) or seen as bijections (corresponding to the
  algebraic point of view).
  For each of them, we define a corresponding first-order logical theory,
  that we call $\mathsf{TOTO}$ (Theory Of Two Orders)
  and $\mathsf{TOOB}$ (Theory Of One Bijection) respectively.
  We consider various expressibility questions in these theories.

  Our main results go in three different direction.
  First, we prove that, for all $k \ge 1$, the set of $k$-stack sortable permutations in the sense of West
  is expressible in $\mathsf{TOTO}$, and that a logical sentence describing this set can be obtained automatically.
  Previously, descriptions of this set were only known for $k \le 3$.
  Next, we characterize permutation classes inside which
  it is possible to express in $\mathsf{TOTO}$ that some given points form a cycle.
  Lastly, we show that sets of permutations that can be described both in $\mathsf{TOOB}$ and $\mathsf{TOTO}$
  are in some sense trivial.
  This gives a mathematical evidence that permutations-as-bijections and permutations-as-words
  are somewhat different objects.
\end{abstract}

{\it Keywords}: permutations, patterns, first order logic, Eurenfest-Fra\"iss\'e games, sorting operators.

\section{Introduction}

This paper being interested in permutations, it should start with a definition of them. 
Some combinatorialists would say that a permutation of size $n$ is a bijection from $[n]= \{1,2, \ldots,n\}$ to itself. 
Others would say that a permutation of size $n$ is a word (sometimes called the \emph{one-line representation} of the permutation) 
on the alphabet $\{1,2, \ldots,n\}$ containing each letter exactly once. 
Both are of course correct. 
The first definition is mostly popular among combinatorialists 
who view permutations as algebraic objects \revision{living} in the permutation group $\S_n$. 
The second one is classical in the \emph{Permutation Patterns} community, where a permutation is also often -- and equivalently -- represented graphically by its \emph{permutation diagram}
which, for $\sigma \in \S_n$ is the $n \times n$ grid containing exactly one dot per row and per column, at coordinates $(i,\sigma(i))_{i \in [n]}$. 

In the huge combinatorics literature on permutations, 
there are, to our knowledge, very few works that consider at the same time 
the algebraic and the pattern view point. 
Some of those we are aware of study the stability of permutation classes by composition of permutations, as in \cite{Karpilovskij}, 
while others are interested in cycles exhibiting a particular pattern structure, see for example \cites{BonaCory,Elizalde} and references therein. 

Although the two (or three -- but words and diagrams are essentially the same thing, as will be clear later in this paper)
definitions do define permutations, they do not give the same point of view on these objects, 
as mirrored by the problems of different nature that are considered on them. 
As remarked by Cameron during his talk at PP2015, Galois even used two different names 
to account for these two definitions of permutations: in Galois's terms, 
a \emph{permutation} was what we described above as a permutation-as-a-word,
while a \emph{substitution} was a permutation-as-a-bijection. 

In this paper, we wish to consider both these definitions of permutations. 
Indeed, these points of view on permutations are believed to be rather orthogonal, and 
one purpose of our paper is to give mathematical evidence that permutations-as-words and permutations-as-bijections 
are really not the same object. 

To that effect, we use the framework of first-order logic. 
For each of these points of view, we define a first-order logical theory,
whose models are the permutations seen as bijections in one case, as words in the other case. 
We then investigate which properties of permutations are expressible in each of these theories. 
While the two theories have appeared briefly in the literature
-- see~\cite[Example 7.6]{Compton:Logical} for a $0$-$1$ law for permutations-as-bijections,
and \cite{cameron2002homogeneous} for the description of the so-called {\em Fra\"iss\'e classes} for permutations-as-words,
it seems that the expressibility of related concepts in the two logics has not been studied in details,
nor has one been compared to the other.
\medskip

It is no surprise that the theory associated with permutations seen
as bijections (called \TOOB -- the Theory Of One Bijection) can express statements about the cycle decomposition of permutations, 
while the theory associated with permutations seen as words (called \TOTO -- the Theory Of Two Orders)
is designed to express pattern-related concepts. 
We will indeed justify in \cref{sec:expressivityTOTO} that 
the containment/avoidance of all kinds of generalized patterns existing in the literature
is expressible in \TOTO. 
But we are also interested in describing which properties are \emph{not} expressible in each of these theories. 
Simple examples of such results that we prove are that \TOTO cannot express that an element of a permutation is a fixed point, 
while \TOOB cannot express that a permutation contains the pattern $231$. 
\bigskip

The results of this article can then be divided into three independent parts. 
\medskip

Our first set of results are expressibility results for \TOTO: 
related to (generalized) patterns, to the substitution decomposition, 
and most interestingly in the context of \emph{sorting operators}.
The study of sorting operators has been one of the historical motivation to study permutation patterns,
after the seminal work of Knuth \cite[Section 2.2.1]{knuth1968art}.
He proved that those permutations sortable with a stack are exactly those avoiding the pattern 231.
This has inspired many subsequent papers considering various sorting operators
(the \emph{bubble sort} operator \cite{chung2010bubble,albert2011bubble},
or queues in parallel \cite{tarjan1972sorting}) or the composition of such sorting operators.
Most notably, the permutations sortable by applying twice or three times the stack sorting operator $\sortS$
have been characterized \cite{west1993sorting,Ulfarsson:3stacks}.
Each of these results has required a generalization of the definition of pattern in a permutation.

In the present paper, we prove that, for each fixed integer $k \ge 1$,
it is possible to express in \TOTO the property that a permutation
is sortable by $k$ iterations of the stack sorting operator 
(\cref{cor:stacksortable}).
Moreover, our proof is constructive (at least in principle) and yields a \TOTO formula
expressing this property.
The same holds for iterations of the bubble sort operator or the queue-and-bypass operator
(or any combination of those; see \cref{prop:Bubble_queue}).

Since \TOTO is the natural logical framework related to patterns,
our theorem gives a ``meta-explanation'' to the fact that many sets of sortable permutations
are described through patterns.
It also suggests that \TOTO is a better framework than pattern-avoidance in this context, 
since it allows the description of $k$-sortable permutations for any fixed value of $k$,
while the cases $k=2$ and $k=3$ had required the introduction of ad hoc notions of patterns which
turn out to correspond to sets definable by certain particular types of formula in \TOTO.

Our second set of results deal with the (in-)expressibility of certain concepts in \TOTO.
As mentioned earlier, it is rather easy to see that 
some simple properties on the cycle structure of permutations 
are not expressible in \TOTO (e.g., the existence of a fixed point -- see \cref{Cor:Fixed_Points_Not_FO}).
It is however possible that they become expressible when we restrict the permutations under consideration to some permutation class $\C$.
As a trivial example, in the class $\C$ which consists only of the increasing permutations $12\cdots n$ (for all $n$) the sentence ``true'' is equivalent to the existence of a fixed point.
In \cref{thm:fixed_point_expressible_in_classes},
we characterize completely the permutation classes $\C$ in which
\TOTO \emph{can} express the fact that a permutation contains a fixed point
(resp. that a given point of a permutation is a fixed point).
We then consider longer cycles and characterize, for any fixed $k$,
the permutation classes $\C$ in which \TOTO can express that $k$ given points form a cycle
(\cref{thm:ExpressCycles}).
On the contrary, the characterization of permutation classes $\C$ in which \TOTO can express that a permutation contains a $k$-cycle 
(but not specifying on which points) is left as an open problem.

The Ehrenfeucht-Fra\"iss\'e theory, giving a game-theoretical characterization of permutations satisfying the same sentences
(see \cref{sec:EFGames}),
together with Erd\H{o}s-Szekeres theorem on the existence of long monotone subsequences in permutations,
are the two key elements in the proof.

Our final results focus on properties of permutations that are expressible
both in \TOTO and \TOOB.
As explained above, the two points of view -- permutations-as-bijection and permutations-as-diagrams --
are believed to be mostly orthogonal, so that we expect that there are few such properties.
We prove indeed (see \cref{Thm:Intersection_True_Or_False}) that these properties are in some sense trivial,
in that they are verified by either all or no permutations with large support
(the support of a permutation being its set of non-fixed points).
This gives the claimed evidence to the fact that permutations-as-words 
and permutations-as-bijections are different objects.
We also give a more precise, and constructive, characterization of
such properties (\cref{Thm:IntersectionBoolean}).
Again, Ehrenfeucht-Fra\"iss\'e games play a central role in the proof of these results.

This study is reminiscent of a result of Atkinson and Beals \cite{atkinson2001permutation}.
They consider permutation classes such that, for each $n$,
the permutations of size $n$ in the class from a subgroup of $\S_n$.
Such classes are proved to belong to a very restricted list.
Again, we observe that asking that a set has some nice property related to patterns (being a class)
and, at the same time, some nice property related to the group structure (being a subgroup),
forces the set under consideration to be somehow trivial.
For a recent refinement of Atkinson-Beals' result, we refer to \cite{Lehtonen2016,lehtonen2017patterngroups}.

We finish this introduction by an open question,
inspired by the important amount of work on $0$-$1$ and convergence laws
in the random graph literature; see, e.g.
\cite[Chapter 10]{JansonRandomGraphs}. 
We also refer to~\cite[Example 7.6]{Compton:Logical} for a $0$-$1$ law for unlabeled models of \TOOB. 
\begin{question}
  For each $n\ge 1$, let $\bm \sigma_n$ be a uniform random permutation of size $n$.
  Does it hold that, for all sentences $\phi$ in \TOTO, the probability that $\bm \sigma_n$
  satisfies $\phi$ has a limit, as $n$ tends to infinity?
\end{question}

We note that there cannot be a $0$-$1$ law for \TOTO since for example the property of being a simple permutation is
expressible in \TOTO and, as shown in \cite{AAKEnumerationSimple}, has limiting probability $1/e^2$.\medskip

{\em Note added in revision:}
a negative answer to the above question was announced by Müller and Skerman
during the 19th International Conference on Random Structures and Algorithms
(Zurich, July '19).
\medskip

The article is organized as follows. 
In Section~\ref{sec:TwoLogics}, we present the two first-order logical theories \TOTO and \TOOB
corresponding to the two points of view on permutations. 
We examine further the expressivity of \TOTO, the theory associated to permutations-as-words, in Section~\ref{sec:expressivityTOTO}. 
In particular, our results on sorting operators are presented in this section.
Next, Section~\ref{sec:EFGames} presents the fundamental tool of Ehrenfeucht-Fra\"iss\'e games,
and shows some first inexpressibility results for \TOTO.
Then, Sections~\ref{Sect:Definability_Cycles} and \ref{Sect:Intersection} go in different directions. 
The first one explores how restricting \TOTO to permutation classes 
allows some properties of the cycle structure of permutations to become expressible. 
The second one describes which properties of permutations are expressible both in \TOTO and in \TOOB. 
\bigskip

\revision{
{\em Notation:}  the following notational convention for integer partitions
is used in Sections~\ref{sec:expressibilityC_laD_la} and \ref{Sect:Intersection}.
Recall that an integer partition $\la$ of $n$ is a nonincreasing list 
of positive integers of sum $n$;
equivalently, it is a finite multiset of positive integers.
We write $(k^m)$ for the partition consisting in $m$ copies of $k$;
also if $\la$ and $\mu$ are two partitions, $\la \cup \mu$ is their disjoint
union as multiset.
In particular $\la \cup (1^k)$ is obtained by adding $k$ parts
equal to $1$ to $\la$.}

%
%
%
\section{Two first-order theories for permutations}\label{sec:TwoLogics}

We assume that the reader has some familiarity with the underlying concepts of first-order logic. 
For a basic introduction, we refer the reader to the Wikipedia page on this topic.
For a detailed presentation of the theory, Ebbinghaus-Flum's book on the topic is a good reference \cite{ModelTheory_Book}.
We will however try to present our two first-order theories
(intended to represent the two points of view on permutations described in the introduction)
so that they are accessible to readers with only a passing familiarity with logic. 

\subsection{Definition of \TOTO and \TOOB}

A \emph{signature} is any set of relation and function symbols, each equipped with an arity (which represents
the number of arguments that each symbol `expects'). 
The symbols in the signature will be used to write formulas. 
In this paper, we consider two signatures, which both have the special property of containing only relation symbols. 

The first one, $\SSS_{\mathsf{TO}}$, consists of two binary relation symbols: $<_P$ and $<_V$. 
It is the signature used in \TOTO, corresponding to viewing permutations as words, or diagrams, 
or more accurately as the data of two total orders indicating the position 
and the value orders of its elements (with $<_P$ and $<_V$ respectively).

The second signature, $\SSS_{\mathsf{OB}}$, consists of a single binary relation symbol, $R$. 
It is the signature used in \TOOB, and a relation between two elements indicates that the first one is sent to the second one by the permutation viewed as a bijection. 

Of course, we intend that $<_P$ and $<_V$ represent strict total orders, and that $R$ represents a bijection; 
this is however not part of the signature. 
It will be ensured later, with the \emph{axioms} of the two \emph{theories} \TOTO and \TOOB. 
To present them, we first need to recap basics about formulas and their models.

\medskip

In the special case of interest to us where there is no function symbol in the signature, 
we can skip the definition of \emph{terms} and move directly to \emph{atomic formulas}. 
These are obtained from variables (taken from an infinite set $\{x,y,z, \dots\}$) by putting them in relation using the relation symbols in the considered signature, or the equality symbol $=$. 
For example, $x=z$, $x <_P y$ and $x <_V x$ are atomic formulas on the signature $\SSS_{\mathsf{TO}}$, 
while  $z=y$, $xRy$ and $xRx$ are atomic formulas on the signature $\SSS_{\mathsf{OB}}$.
As is usual for binary relations we write them infix as above rather than the more proper $<_P(x,y)$

\emph{First-order formulas}, usually denoted by Greek letters ($\phi, \dots$) are then obtained inductively from the atomic formulas, 
as combinations of smaller formulas using the usual connectives of the first-order logic:
negation ($\neg$), 
conjunction ($\wedge$),
disjunction ($\vee$),
implication ($\rightarrow$),
equivalence ($\leftrightarrow$),
universal quantification ($\forall x \, \phi$, for $x$ a variable and $\phi$ a formula)
and existential quantification ($\exists x \, \phi$). 
Note that we restrict ourselves to first-order formulas: quantifiers may be applied only to single variables 
(as opposed to sets of variables in second-order for instance). 

A \emph{sentence} is a formula that has no free variable, that is to say in which all variables are quantified. 
For example, $\phi_1 = \exists x \ \exists y \ (x <_P y \et y <_V x)$ is a sentence on the signature $\SSS_{\mathsf{TO}}$ 
and $\phi_2(x) = \exists y\ xRy \et yRx $ is not a sentence but a formula with one free variable $x$ on the signature $\SSS_{\mathsf{OB}}$. 


The formulas themselves do not describe permutations. 
Instead, formulas can be used to describe properties of permutations. 
Permutations are then called \emph{models} of a formula. 

Generally speaking, given a signature $\SSS$, a \emph{(finite) model} is a pair $\MMM = (A,I)$ 
where $A$ is any finite set, called the \emph{domain}, 
and $I$ describes an interpretation of the symbols in $\SSS$ on $A$. 
Formally, $I$ is the data, for every relation symbol $R$
(say, of arity $k$) in the signature, of a subset $I(R)$ of $A^k$.
Note that models with infinite domains also exist; 
however, in this paper, we restrict ourselves to finite models. 
This makes sense since models are intended to represent permutations of any finite size.

Two models $\MMM = (A,I)$ and $\MMM' = (A',I')$ are \emph{isomorphic} when 
there exists a bijection $f$ from $A$ to $A'$ such that 
for every relation symbol $R$ (say, of arity $k$) in the signature, 
for all $k$-tuples $(a_1, \dots, a_k)$ of elements of $A$, 
$(a_1, \dots, a_k)$ is in $I(R)$ if and only if $(f(a_1), \dots, f(a_k))$ is in $I'(R)$ 
(together with a similar condition for function symbols, if the signature contains some).

For example, on the signature $\SSS_{\mathsf{TO}}$, a model is 
\[
\MMM_{\mathsf{TO}} =\left(\{a,b,c,d,e\},
\begin{cases}
<_P \mapsto \prec_P \\
<_V \mapsto \prec_V 
\end{cases} \right),
\]
where $\prec_P$ (resp. $\prec_V$) is the strict total order defined on $A=\{a,b,c,d,e\}$ by $a \prec_P b \prec_P c \prec_P d \prec_P e$ 
(resp. $c \prec_V e \prec_V a \prec_V d \prec_V b$). 
This model is isomorphic to \[
\MMM'_{\mathsf{TO}} =\left(\{1,2,3,4,5\},
\begin{cases}
<_P \mapsto \prec'_P \\
<_V \mapsto \prec'_V 
\end{cases} \right),
\]
where $\prec'_P$ and $\prec'_V$ are the strict total orders defined on $\{1,2,3,4,5\}$ by $1 \prec'_P 2 \prec'_P 3 \prec'_P 4 \prec'_P 5$ 
and $3 \prec'_V 5 \prec'_V 1 \prec'_V 4 \prec'_V 2$, the underlying bijection $f$ being $a\mapsto 1, \dots, e \mapsto 5$. 

As we will explain in more details later (see Section~\ref{sec:perm_as_models}), 
$\MMM_{\mathsf{TO}}$ and $\MMM'_{\mathsf{TO}}$ both represent the permutation $\sigma$ which can be written in one-line notation as $\sigma = 35142$. 
This same permutation, which decomposes into a product of cycles as $\sigma = (1,3)(2,5)(4)$, 
can of course also be represented by a model on the signature $\SSS_{\mathsf{OB}}$: for example 
$\MMM_{\mathsf{OB}} = (\{1,2,3,4,5\},\mathcal{R})$ where the only pairs in $\mathcal{R}$ are $1\mathcal{R}3$, $2\mathcal{R}5$, $3\mathcal{R}1$, $4\mathcal{R}4$ and $5\mathcal{R}2$.

Note that we have been careful above to use different notations for the relation symbols ($<_P,<_V,R$) and their interpretations ($\prec_P,\prec_V,\mathcal{R}$); 
in the following, we may be more flexible and use the same notation for both the relation symbol and its interpretation. 

Finally, a model $\MMM= (A,I)$ is said to \emph{satisfy} a sentence $\phi$ when the truth value of $\phi$ is ``True" when interpreting all symbols in $\phi$ according to $I$. 
We also say that $\MMM$ is a model of $\phi$, and denote it by $\MMM \models \phi$. 
For example, it is easily checked that $\MMM_{\mathsf{TO}}$ above is a model of our earlier example formula $\phi_1$ 
(since, for instance, evaluating $x$ to $a$ and $y$ to $e$ makes the inner formula $x <_P y \et y <_V x$ true).

In the special case of interest to us where all models are finite 
two models satisfy the same sentences if and only if they are isomorphic. 

The definition of satisfiability above applies only to sentences, which do not have free variables. 
It can however be extended to formulas with free variables, provided that the free variables are \emph{assigned} values from the domain. 
We usually write $\phi(\mathbf{x})$ a formula with free variables $\mathbf{x} = (x_1, \dots, x_k)$, 
and, given $\MMM= (A,I)$ a model and $\mathbf{a} = (a_1, \dots, a_k) \in A^k$, 
we say that $(\MMM,\mathbf{a})$ satisfies $\phi(\mathbf{x})$ (written $(\MMM,\mathbf{a}) \models \phi(\mathbf{x})$)
if the truth value of $\phi(\mathbf{x})$ is ``True'' when every $x_i$ is interpreted as $a_i$ and relations symbols are interpreted according to $I$. 
Getting back to our examples
\[
(\MMM_{\mathsf{OB}},(1)) \models \exists y \, \left( xRy \et yRx \right)
\]
since we can witness the existential quantifier with $y = 3$ and it is the case that $1\mathcal{R}3$ and $3\mathcal{R}1$.

\medskip

Formally speaking, a \emph{theory} is then just a set of sentences, which are called the \emph{axioms} of the theory. 
A \emph{model of a theory} is any model that satisfies all axioms of the theory. 

The axioms of a theory can be seen as ensuring some properties of the models considered, 
or equivalently as imposing some conditions of the interpretations of the relation symbols. 
In our case, the axioms of \TOTO (the Theory Of Two Orders) ensure that $<_P$ and $<_V$ are indeed strict total orders, 
while the axioms of \TOOB (the Theory Of One Bijection) indicate that $R$ is a bijection. 
The fact that these properties can be described by first-order sentences (to be taken then as the axioms of our theories) 
is an easy exercise that is left to the reader. 
This completes the definition of our two theories \TOTO and \TOOB. 

\subsection{Permutations as models of \TOTO and \TOOB} \label{sec:perm_as_models}

We now turn to explaining more precisely how permutations can be encoded as models of \TOTO and \TOOB.

\medskip

We start with \TOOB.
Given a permutation $\sigma$ of size $n$, the model of \TOOB that we associate to it is 
$(A^\sigma, R^{\sigma})$, where $A^\sigma = \{1,2,\dots, n\}$ and $R^{\sigma}$ is defined by 
$iR^{\sigma}j$ if and only if $\sigma(i)=j$. 

It should be noticed that the total order on $\{1,2,\dots, n\}$ is not at all captured by the model $(A^\sigma, R^{\sigma})$. 
\TOOB is on the contrary designed to describe properties of the cycle decomposition of permutations. 
For instance, the existence of a cycle of a given size is very easy to express with a sentence of \TOOB:
the existence of a fixed point is expressed by $\exists\, x\ xRx$;
while the existence of a cycle of size $2$ is expressed by
\[
\exists\, x,y \, \left( x \ne y \, \wedge \, xRy\,\wedge\, yRx \right);
\]
and the generalization to cycles of greater size is obvious.
It should also be noted that \TOOB can only express ``finite'' statements: 
for any given $k$, it is expressible that a permutation has size $k$ and is a $k$-cycle (using $k+1$ variables), 
but it is not possible to express that a permutation consists of a single cycle (of arbitrary size).

Although not all models of \TOOB are of the form $(A^\sigma, R^{\sigma})$ for a permutation $\sigma$,
the following proposition shows that it is almost the case.

\begin{proposition}
For any model $(A,R)$ of \TOOB, there exists a permutation $\sigma$ such that $(A,R)$ and $(A^\sigma, R^{\sigma})$ are 
isomorphic. 
In this case, we say that $\sigma$ is a permutation associated with $(A,R)$. 
\end{proposition}

\begin{proof}
Let $(A,R)$ be a model of \TOOB, and denote by $n$ the cardinality of $A$.
Consider any bijection $f$ between $A$ and $\{1,2,\dots, n\}$, and let $\sigma$ be the permutation such that 
$\sigma(i) = j$ if and only if $f^{-1}(i) R f^{-1}(j)$.
Clearly, $(A,R)$ and $(A^\sigma, R^{\sigma})$ are isomorphic.
\end{proof}

It is already visible in the above proof that the permutation associated to a given model of \TOOB is not uniquely defined. 
The next proposition describes the relation between all such permutations. 
Recall that the \emph{cycle-type} of a permutation of size $n$ is the partition $\lambda = (\lambda_1, \dots, \lambda_k)$ of $n$ 
such that $\sigma$ can be decomposed as a product of $k$ disjoint cycles, of respective sizes $\lambda_1, \dots, \lambda_k$.
Recall also that two permutations are \emph{conjugate} if and only if they have the same cycle-type. 
The following proposition follows easily from the various definitions.

\begin{proposition}
\label{prop:sameTOOB_conjugatePerm}
Two models of \TOOB are isomorphic if and only if the permutations associated with them are conjugate. 
\end{proposition} 

For finite models, being isomorphic is equivalent to satisfying the same set of sentences.
In particular, the models corresponding to conjugate permutations satisfy the same sentences.
This proves our claim of the introduction that the avoidance of $231$ is not expressible in \TOOB, since $231$ and $312$ are conjugate
and of course one of them contains $231$ while the other doesn't! 
This shows a weakness of the expressivity of \TOOB. 
In particular, the containment of a given pattern is in general not expressible in \TOOB 
(unlike in \TOTO, as we discuss in Section~\ref{sec:patterns}).
In fact, it is easy to check that, \revision{for $|\pi| \ge 3$,}
it is not possible to express the avoidance of $\pi$ in \TOOB.


We now move to \TOTO. 
As we shall see, the focus on permutations is different: 
\TOTO considers the relative order between the elements of the permutations
and does not capture the cycle structure (see Corollary \ref{Cor:Fixed_Points_Not_FO} and the more involved \cref{thm:ExpressCycles}).

To represent a permutation $\sigma$ of size $n$ as a model of \TOTO, we do the following. 
We consider the domain $A^\sigma = \{(i,\sigma(i)) : 1 \leq i \leq n\}$ of cardinality $n$, 
and we encode $\sigma$ by the triple $(A^\sigma, <_P^{\sigma}, <_V^{\sigma})$ 
where $<_P^{\sigma}$ (resp. $<_V^{\sigma})$) is the strict total order on $A^\sigma$ 
defined by the natural order on the first (resp. second) component of the elements of $A^\sigma$. 
When it is clear from the context (as for instance in Proposition~\ref{prop:sameFO_samePerm}), 
we may denote the model $(A^\sigma, <_P^{\sigma}, <_V^{\sigma})$ of \TOTO simply by $\sigma$.

Representing a permutation $\sigma$ by $(A^\sigma, <_P^{\sigma}, <_V^{\sigma})$ as above 
is very close to the representation of permutations by their diagrams. 
When considering permutations \emph{via} their diagrams, it is often observed that the actual coordinates of the points do not really matter, but only their relative positions. 
Hence, a property that we would like to hold is that 
two isomorphic 
models of \TOTO 
should represent the same permutation. 
To make this statement precise, we need to define the permutation that is associated with a model of \TOTO. 
%

Let $(A,<_P,<_V)$ be any model of \TOTO. 
First, for any strict total order $<$ on $A$, we define the \emph{rank} of $a \in A$ as 
$\rank(a) = \card \{b \in A : b < a\} +1$. 
Note that when $a$ runs over $A$ and $A$ is of cardinality $n$, $\rank(a)$ takes exactly once each value in $\{1,2, \ldots,n\}$. 
Then, we write $A= \{a_1, a_2,  \ldots, a_n\}$ where each $a_i$ has rank $i$ for $<_P$, and 
we let $r_i$ be the rank of $a_i$ for $<_V$, for each $i$. 
The permutation $\sigma$ is the one whose one-line representation is $r_1 r_2 \ldots r_n$. 

It is obvious that $(A,<_P,<_V)$ and $(A^\sigma, <_P^{\sigma}, <_V^{\sigma})$ are isomorphic. 
Moreover, if two models of \TOTO are isomorphic, then the permutations associated with them are equal. 
We therefore have the following analogue in \TOTO of Proposition~\ref{prop:sameTOOB_conjugatePerm} for \TOOB.

\begin{proposition}
\label{prop:sameFO_samePerm}
Two models of \TOTO are isomorphic and, hence, satisfy the same set of sentences
if and only if the permutations associated with them are equal. 
\end{proposition}

Comparing Propositions~\ref{prop:sameTOOB_conjugatePerm} and~\ref{prop:sameFO_samePerm}, 
we immediately see that \TOTO allows to describe a lot more details than \TOOB. 
Indeed, while formulas of \TOTO allow to discriminate all permutations among themselves, 
formulas of \TOOB do not distinguish between permutations of the same cycle-type. 
The expressivity of \TOTO (and to a lesser extent, of \TOOB) will be further discussed in the rest of the paper.

\section{Expressivity of \TOTO} \label{sec:expressivityTOTO}

As we explained earlier, \TOTO 
has been designed to express pattern-related concepts in permutations. 
In this section, we illustrate this fact with numerous examples. 
\begin{itemize}
  \item In \cref{sec:patterns},
    we consider containment/avoidance of various kind of generalized patterns.
  \item In \cref{ssec:subst}, we investigate some properties linked to the substitution decomposition, namely 
    being $\oplus$/$\ominus$ indecomposable and being simple.
  \item \cref{ssec:sorting} considers expressibility results in the context of sorting operators.
    It contains in particular our first main result, on the expressibility
    of the set of $k$-stack sortable permutations, in the sense of West.
  \item \cref{sec:expressibilityC_laD_la} shows some properties related to the cycle structure
    which are nevertheless expressible in \TOTO. This is a preparation for \cref{Sect:Intersection}.
\end{itemize}

\subsection{Notions of patterns} \label{sec:patterns}

We assume that the reader is familiar with basic concepts related to patterns in permutations. 
Let us only recall a couple of very classical definitions. 
A permutation $\sigma = \sigma(1)\sigma(2) \dots \sigma(n)$ \emph{contains} the permutation $\pi= \pi(1) \pi(2) \dots \pi(k)$ as a \emph{(classical) pattern} 
if there exist indices $i_1< i_2 <\dots< i_k$ between $1$ and $n$ such that $\pi(j) < \pi(h)$ if and only if $\sigma(i_j) < \sigma(i_h)$. In this case we also 
say that $\pi$ \emph{is a pattern of} $\sigma$.
The sequence $(i_1, i_2, \dots, i_k)$ is called an \emph{occurrence} of $\pi$ in $\sigma$. If $\sigma$ does not contain $\pi$ then we say that $\sigma$ \emph{avoids} $\pi$.
A \emph{permutation class} is a set $\C$ of permutations such that, if $\sigma \in \C$ and $\pi$ is a pattern of $\sigma$, then $\pi \in \C$. We note that ``contains'' is a partial order on the set of all finite permutations.
Equivalently, 
permutation classes can be characterized as those sets of permutations that avoid some (possibly infinite) family $B$ of patterns
which we write $\C = \Av(B)$.
If no permutation of $B$ contains any other as a pattern, then this determines $B$ uniquely from $\C$ and $B$ is called the \emph{basis} of $\C$.
Other classical definitions about permutations and their patterns have been conveniently summarized 
in Bevan's brief presentation~\cite{Bevan:brief} prepared for the conference \emph{Permutation Patterns 2015}. 
The goal of this section is to illustrate that all notions of patterns
that we have found in the literature are expressible in \TOTO. 

\subsubsection{Classical patterns}

Note that, when defining that $(i_1, i_2, \dots, i_k)$ is an occurrence of $\pi$ in $\sigma$,
we request the indices $i_j$ to be increasing. 
It could also be natural to consider any permutation of $(i_1, i_2, \dots, i_k)$  as an occurrence of $\pi$,
but this is {\em not} what we are doing here. 
Also, in our framework, if a sequence $(i_1, i_2, \dots, i_k)$ of indices gives an occurrence of $\pi$, 
it is convenient to write that $\big((i_1,\sigma(i_1)),(i_2,\sigma(i_2)), \dots, (i_k,\sigma(i_k)) \big)$ is an occurrence of $\pi$ in $\sigma$. 
%
The same remarks apply to more general types of patterns, discussed in Section~\ref{sec:generalized_patterns}.

\begin{proposition}
\label{prop:classical_pattern}
Let $\pi$ be any permutation of size $k$. 
There exists a formula $\psi_\pi(x_1, \ldots, x_k)$ of \TOTO to express the property
that $k$ elements $a_1, \ldots, a_k$ of a permutation $\sigma$ 
form an occurrence of the pattern $\pi$ (in the classical sense);
more precisely, for any permutation $\sigma$ and elements $a_1, \ldots, a_k$ in $\sigma$,
\[
(a_1, \ldots, a_k) \text{ is an occurrence of $\pi$ in } \sigma 
\Leftrightarrow (\sigma,a_1, \ldots, a_k) \models \psi_\pi(x_1, \ldots, x_k).
\]
\end{proposition}
For example, the pattern $231$ corresponds to the formula 
\[ \psi_{231}(x,y,z) := (x <_P y <_P z) \et (z <_V x <_V y).\] 

\begin{proof}
Clearly, generalizing the above example, it is enough to take the following formula, where $k$ is the size of $\pi$: 
\[
\footnotesize{
 \psi_\pi(x_1,\dots,x_k) = (x_1 <_P x_2 <_P \dots <_P x_k) \et (x_{\pi^{-1}(1)} <_V x_{\pi^{-1}(2)} <_V \dots <_V x_{\pi^{-1}(k)}).} \qedhere
\]
\end{proof}

\begin{corollary}
\label{cor:classical_pattern_sentence}
The containment of a given pattern is a property expressible in \TOTO. 
\end{corollary}

\begin{proof}
  Let $\pi$ be a given pattern of size $k$. The containment of $\pi$ simply corresponds to the formula
  \[ \exists \, x_1,\dots,x_k \, \psi_\pi(x_1,\dots,x_k). \qedhere\]
\end{proof}

Taking negations and conjunctions, Corollary~\ref{cor:classical_pattern_sentence} immediately gives the following. 
\begin{corollary}
The avoidance of a given pattern, 
and membership of any finitely-based permutation class, are properties expressible in \TOTO. 
\end{corollary}

\begin{remark}
From a formula expressing that some sequence of elements in some permutation forms a given pattern $\pi$, 
we can always write a sentence to express the existence of an occurrence of $\pi$ 
(by introducing existential quantifiers). 
This has been used in Corollary~\ref{cor:classical_pattern_sentence} above in the case of classical patterns, 
but applies just as well to all other notions of patterns below. 
Therefore, we will only state the next results for specific occurrences (i.e., the analogue of Proposition~\ref{prop:classical_pattern}), 
leaving the existence of occurrences (i.e., the analogue of Corollary~\ref{cor:classical_pattern_sentence}) to the reader.
\end{remark}

\subsubsection{Generalized notions of patterns} \label{sec:generalized_patterns}

Several generalizations of the notion of classical patterns exist. 
The most common ones are recorded in~\cite{Bevan:brief},
and an overview with more such generalizations can be found in~\cite[Figure 3]{Ulfarsson:3stacks}. 
In the following, we use the notation (and some examples) of~\cite{Ulfarsson:3stacks}, and refer to this paper for the formal definitions. 

Among the earliest generalizations of patterns, some (like vincular or bivincular patterns)
indicate that some of the elements forming an occurrence of the underlying classical pattern must be consecutive (for $<_P$ or for $<_V$). 
This is further generalized by the notion of mesh patterns;
to form an occurrence of a mesh pattern, elements must
\begin{itemize}
  \item first form an occurrence of the corresponding classical pattern;
  \item in addition, some of the regions determined by these elements (whose set is recorded in a list $R$) 
    should be empty (see example below).
\end{itemize}
Such constraints are easily expressible in \TOTO. 

\begin{proposition}
\label{prop:mesh}
Let $(\pi,R)$ be any mesh pattern, with $\pi$ of size $k$. 
There exists a formula $\psi_{(\pi,R)}(x_1,\dots,x_k)$ of \TOTO 
to express the property that elements $(a_1,\dots,a_k)$ of a permutation $\sigma$ form an occurrence of the mesh pattern $(\pi,R)$.
\end{proposition}

For instance, 
the formula corresponding to occurrences of the mesh pattern 
$(132, \{(0,2),(1, 2),(2,2)\}) = 
\begin{array}{c}\begin{tikzpicture}[scale=0.2] 
\draw[color=white, pattern=north east lines, pattern color=black] (0,2) rectangle (3,3);
\draw (0,1) -- (4,1);
\draw (0,2) -- (4,2);
\draw (0,3) -- (4,3);
\draw (1,0) -- (1,4);
\draw (2,0) -- (2,4);
\draw (3,0) -- (3,4);
\fill (1,1) circle (.3);
\fill (2,3) circle (.3);
\fill (3,2) circle (.3);
\end{tikzpicture}
\end{array}$ 
is
{\footnotesize
\begin{align*}
\psi_{132}(x,y,z) &\et \neg \exists t \, ( t <_P x \et z <_V t <_V y) \\
& \et \neg \exists t \, ( x <_P t <_P y \et z <_V t <_V y) \et \neg \exists t \, ( y <_P t <_P z \et z <_V t <_V y).
\end{align*}
}
\begin{proof}
The idea of the previous example generalizes immediately as follows:
{\footnotesize
\begin{align*}
 \psi_{(\pi,R)}(x_1,\dots,x_k) = & \psi_\pi(x_1,\dots,x_k) \\
 & \et \bigwedge_{(i,j) \in R} \neg \exists t \, ( x_i <_P t <_P x_{i+1} \et x_{\pi^{-1}(j)} <_V t <_V x_{\pi^{-1}(j+1)}),
\end{align*}}
where the possible comparisons with $x_0$, $x_{k+1}$, $x_{\pi^{-1}(0)}$ or $x_{\pi^{-1}(k+1)}$ in the last part of the formula are simply dropped.
\end{proof}

Mesh patterns can be further generalized, as decorated patterns. 
This is one of the three most general types of patterns that appear in the overview of~\cite[Figure 3]{Ulfarsson:3stacks}, 
the other two being barred patterns and ``grid classes''. 
We refer to~\cite{Ulfarsson:3stacks} for the definition of barred patterns and to~\cite[Section 4.3]{Vatter:survey} for the definition of grid classes (called generalized grid classes therein). 
It is easy to generalize the ideas of the previous proofs
to show that all properties of containing such patterns are expressible in \TOTO, 
although it becomes notationally more painful to write it in full generality. 
We therefore only record the following results, without proof, but illustrated with examples. 

\begin{proposition}
Let $(\pi,\bm{\mathcal{C}})$ be any decorated pattern. 
There exists a \TOTO formula with $k = |\pi|$ free variables 
to express the property that elements $(a_1,\dots,a_k)$ of a permutation
$\sigma$ form an occurrence of $(\pi,\bm{\mathcal{C}})$.
\end{proposition}

For instance, an occurrence of the decorated pattern $\textsf{dec} = (21, \{((1,1),12)\} ) =
\begin{array}{c}\begin{tikzpicture}[scale=0.25] 
\draw[color=white, pattern=north east lines, pattern color=black] (1,1) rectangle (2,2);
\draw (0,1) -- (3,1);
\draw (0,2) -- (3,2);
\draw (1,0) -- (1,3);
\draw (2,0) -- (2,3);
\fill (1,2) circle (.3);
\fill (2,1) circle (.3);
\fill (1.3,1.3) circle (.13);
\fill (1.6,1.6) circle (.13);
\end{tikzpicture}
\end{array}$
is an occurrence $(a,b)$ of the classical pattern $21$
such that the middle region delimited by $a$ and $b$ does not contain a pattern $12$.
This corresponds to the \TOTO formula 
\[
\psi_{21}(x,y) \et \neg \Big(\exists t \exists u\, \big[ (a <_P t,u <_P b) \et (b <_V t,u <_V a) \big] 
\et \psi_{12}(t,u)\Big).
\]
Here, $(a <_P t,u <_P b)$ is a short notation for $(a <_P t <_P b)  \et (a <_P u <_P b)$.
We will use similar abbreviations in the following.

\begin{proposition}
Let $\pi$ be any barred pattern. 
There exists a \TOTO formula, 
whose number $k$ of free variables is the number of unbarred elements in $\pi$,
to express the property that
elements $(a_1,\dots,a_k)$ of a permutation
$\sigma$ form an occurrence of $\pi$.
\end{proposition}

Consider for instance the barred pattern $\bar{1}3\bar{2}4$
(which has been chosen since it cannot be expressed using decorated patterns, see~\cite{Ulfarsson:3stacks}). 
Informally, an occurrence of $\bar{1}3\bar{2}4$ is a pair $(b,d)$ of two elements in increasing order 
(i.e., in the order induced by the unbarred elements),
which cannot be extended to a quadruple $(a,b,c,d)$ that is an occurrence of $1324$.

It is easy to see that the following formula expresses occurrences of the barred pattern $\bar{1}3\bar{2}4$:
{\footnotesize
\[
\psi_{12}(x,y) \et \neg \Big(\exists u \exists t\, \psi_{1324}(u,x,t,y) \Big).
\]
}

\begin{proposition}
  \label{prop:grid}
Let $M$ be a matrix whose entries are finitely-based permutation classes. 
There exists a \TOTO sentence $\phi_{M}$ to express the property of membership to the grid class $\Grid(M)$.
\end{proposition}

For instance, consider the grid class $\Grid(M)$ for the matrix 
$$M= \begin{pmatrix}
\Av(123) & \emptyset \\
\Av(21) & \Av(12) \end{pmatrix}.$$
Informally a permutation $\sigma$ is an element of $\Grid(M)$
if its diagram can be split in 4 regions (by one vertical and one horizontal lines that do not cross any dots)
such that 
\begin{itemize}
  \item the NW (resp. SW, resp. SE) region does not contain any occurrence of $123$ (resp. $21$, resp. $12$); 
  \item the NE corner is empty.
\end{itemize}
It is easy to see that the following sentence $\tilde{\phi}_{M}$ indicates the membership to $\Grid(M)$ with {\em SW-non-trivial decomposition}
(i.e. there is at least one point on the left of the vertical line and below the horizontal line that appear in the above definition).
Modifying it to get a sentence $\phi_M$ that indicates membership to $\Grid(M)$ is a straightforward exercise.
\begin{align*}
  \tilde{\phi}_{M}=
 \exists \ell_v \exists \ell_h \, &
 \neg \Big(\exists x \exists y \exists z \, \big[(x,y,z \leq_P \ell_v) \et (x,y,z >_V \ell_h) \big] \et \psi_{123}(x,y,z)\Big)\\
\et &\neg \Big(\exists x \exists y \, \big[(x,y \leq_P \ell_v) \et (x,y \leq_V \ell_h) \big] \et \psi_{21}(x,y) \Big)\\
\et &\neg \Big(\exists x \exists y \, \big[(x,y >_P \ell_v) \et (x,y \leq_V \ell_h)\big] \et \psi_{12}(x,y) \Big) \\
\et & \neg \Big(\exists x \, \big[ (x >_P \ell_v) \et (x >_V \ell_h) \big]\Big).
\end{align*}
The variables $\ell_v$ and $\ell_h$ correspond to the vertical and horizontal lines in the definition of grid classes.
Since we cannot quantify on lines that do not cross any dot, $\ell_v$ (resp. $\ell_h$) is the rightmost dot on the left of the vertical line 
(resp. the highest dot below the horizontal line). 
Note that, by definition, both dots exist in SW-non-trivial decompositions. 

\begin{remark}
  Consider a matrix $M$ such that each entry $M_{i,j}$ is a set of permutations 
  for which membership is expressible by a \TOTO sentence (e.g. containing/avoiding any given generalized pattern).
  Then membership to the set $\Grid(M)$ is expressible by a \TOTO sentence.

  These generalized {\em grid sets} (they are not {\em permutation classes} anymore) do not seem to have been introduced in the literature. 
  However, we want to point out that they provide a framework which generalizes both grid classes and
  avoidance of barred/decorated patterns and which still fits in the expressivity range of \TOTO.
\end{remark}

\subsection{Concepts related to substitution decomposition}\label{ssec:subst}

Substitution decomposition is an approach to the study of permutations and permutation classes 
which has proved very useful, in particular, but not only, for enumeration results.
We refer to \cite[Section 3.2]{Vatter:survey} for a survey on this technique.
In this section, we prove that standard concepts related to it are expressible in \TOTO.
We will also use the inflation operation defined below later in the paper.

Given $\pi$ a permutation of size $k$ and $k$ permutations $\sigma_1, \dots, \sigma_k$, the permutation 
$\pi[\sigma_1, \dots, \sigma_k]$ (called \emph{inflation of $\pi$ by $\sigma_1, \dots, \sigma_k$}) 
is the one whose diagram is obtained from that of $\pi$ by replacing each point $(i,\pi(i))$ by the diagram of $\sigma_i$ 
(which is then referred to as a block). 
An example is given in \cref{fig:Inflation}.
\begin{figure}[t]
 \[ \includegraphics[height=2.3cm]{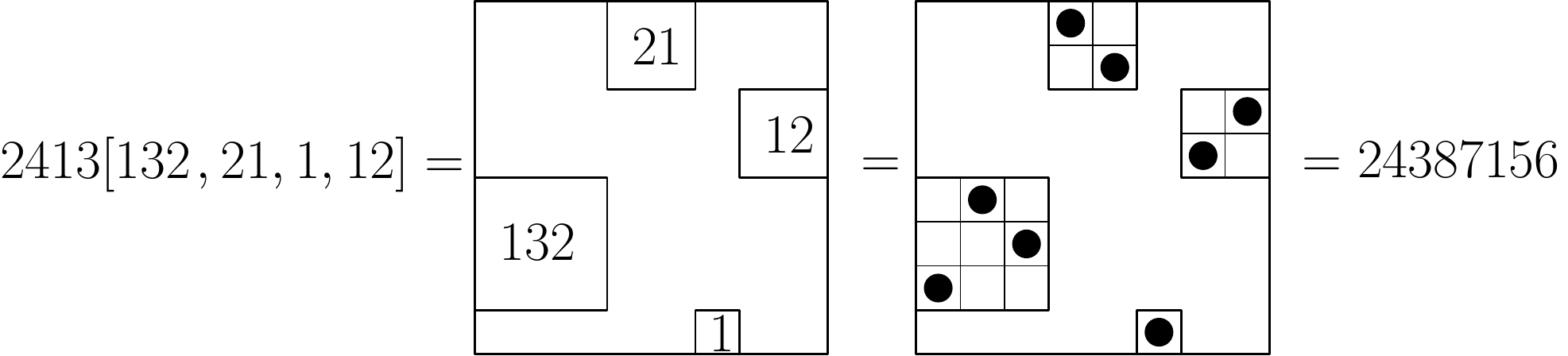}\]
  \caption{An example of inflation.}
  \label{fig:Inflation}
\end{figure}

A permutation that can be written as $\pi[\sigma_1, \dots, \sigma_k]$ is called \emph{$\pi$-decomposable}. 
When $\pi$ is increasing (resp. decreasing), we may write $\oplus$ (resp. $\ominus$) instead of $\pi$. 

\begin{proposition}
There exists a \TOTO sentence $\phi_\oplus$ to express the property that any given permutation is $\oplus$-decomposable. 
\end{proposition}
Obviously, the same holds for $\ominus$-decomposable permutations.
The proof actually also easily extends to show that being an inflation of $\pi$, for any given $\pi$, 
is expressible by a \TOTO sentence. 

\begin{proof}
It is enough to take $\phi_\oplus$ to be the following sentence 
\[
\phi_\oplus= \exists \ell_v \exists \ell_h \,  \Big( \exists x \, [x >_P \ell_v]\Big)  
\et  \Big( \forall x \, \big[ (x \leq_P \ell_v) \leftrightarrow (x \leq_V \ell_h) \big]\Big).
\]
Note that the part $( \exists x \, [x >_P \ell_v] )$ in the above sentence ensures that 
the second block in the $\oplus$-decomposition is not empty.
(For the first block, this is automatically ensured by the existence of $\ell_v$, for instance). 
We could have used equivalently $( \exists x \, [x >_V \ell_h] )$.
\end{proof}
\begin{remark}
  This is a variant of the statement and proof on grid classes (\cref{prop:grid}),
  where the decomposition has to be non trivial.
\end{remark}

%
%

Simple permutations, i.e. permutations that cannot be obtained as a non-trivial inflation,
play a particularly important role in substitution decomposition.
They can be alternatively characterized as permutations
which do not contain a non-trivial interval, an interval
being a range of integers sent to another range of integers
by the permutation.
They can also be described in \TOTO. 

\begin{proposition}
\label{prop:simple}
There exists a \TOTO sentence $\phi_{simple}$ to express the property that any given permutation is simple. 
\end{proposition}

\begin{proof}
  We simply take $\phi_{\text{simple}} = \neg \phi_{\text{int}}$, 
  where $\phi_{\text{int}}$ indicates the existence of a non-trivial interval
and is given as follows
\begin{align*}
  \phi_{\text{int}}:= \exists \ell_{v,1}& \exists \ell_{v,2} \exists \ell_{h,1} \exists \ell_{h,2}
[(\ell_{v,1} <_P \ell_{v,2}) \et \Big(\exists y [y<_P \ell_{v,1}] \vee [\ell_{v,2} <_P y]\Big)\\
\et & \Big(\forall x \, \big[ (\ell_{v,1} \leq_P x \leq_P \ell_{v,2}) 
\leftrightarrow (\ell_{h,1} \leq_V x \leq_V \ell_{h,2}) \big]\Big). 
\end{align*}
Indeed, the second line indicates that there is an interval in the permutation between
the horizontal lines $\ell_{h,1}$ and $\ell_{h,2}$ and the vertical lines $\ell_{v,1}$ and $\ell_{v,2}$
(the dots corresponding to all these lines being included in the interval).
The first line then indicates that the interval is non-trivial.
\end{proof}


Note that Proposition~\ref{prop:simple} can also be seen as a corollary of Proposition~\ref{prop:mesh}, 
since being simple is expressible by the avoidance of a finite set of mesh patterns (see~\cite[Proposition 2.1]{Ulfarsson:3stacks}). 

\subsection{Sorting operators}\label{ssec:sorting}

Broadly speaking a sorting operator is nothing more than a function $S$ from the set of all permutations to itself that preserves size. 
We denote by $\I$ the set $\{12\dots n , n\geq 1\}$ of increasing monotone permutations. 
The \emph{$S$-sortable permutations} are just the permutations whose image under $S$ is an increasing permutation, i.e., $S^{-1}(\I)$. More generally, for a positive integer $k$ we call the set $S^{-k}(\I)$ the \emph{$S$-$k$-sortable permutations}. Of course if we are interested in actually using $S$ as an effective method of sorting it should be the case that we can guarantee that every permutation $\sigma$ is $S$-$k$-sortable for some $k$. One sufficient condition for this that frequently holds (and which could be taken as part of the definition of sorting operator) is that the set of inversions of $S(\sigma)$ should be  a subset of those of $\sigma$ with the inclusion being strict except on $\I$. Alternatively one could simply require that the number of inversions of $S(\sigma)$ should be less than the number of inversions of $\sigma$.

Suppose that $S$ is a sorting operator and, just for example, that $S(42531) = 24135$. 
What has really happened is that the positional order, $\prec_P$, of the original permutation ($4 \prec_P 2 \prec_P 5 \prec_P 3 \prec_P 1$, where we take the value order to be $1 \prec_V \dots \prec_V 5$) 
has been replaced by a new one, $\prec_{S(P)}$, where $2 \prec_{S(P)} 4 \prec_{S(P)} 1 \prec_{S(P)} 5 \prec_{S(P)} 3$ (keeping $\prec_V$ unchanged). 
So, another view of sorting operators is that they simply change the positional order of a permutation, i.e., that they are isomorphism-preserving maps taking permutations $(X, \prec_P, \prec_V)$ to other permutations $(X, \prec_{S(P)}, \prec_V)$. 

We will be particularly interested in the case where $\prec_{S(P)}$ can be expressed by a formula in \TOTO. 
We say that a sorting operator $S$ is \emph{\TOTO-definable} if there is a \TOTO formula $\phi_{SP} (x,y)$ with two free variables such that 
for all permutations $\sigma=(X, \prec_P, \prec_V)$ and all $a, b \in X$, $a \prec_{S(P)} b$ if and only if $(\sigma,a,b) \models \phi_{SP}(x,y)$.
Similarly, we say that a set $\T$ of permutations is \emph{definable in \TOTO} when there exists a \TOTO sentence whose models are exactly the permutations in $\T$. 
\begin{proposition}
\label{prop:S-sortable_definable}
Suppose that a sorting operator, $S$, is \TOTO-definable. 
Then, for any set of permutations $\T$ definable in \TOTO, $S^{-1}(\T)$ is also definable in \TOTO.
In particular, for any fixed $k$, there exists a \TOTO sentence
whose models are exactly the $S$-$k$-sortable permutations. 
\end{proposition}

\begin{proof}
Let $\T$ be a set of permutations definable in \TOTO
and take any sentence $\phi_{\T}$ that defines $\T$.
Replacing every occurrence of $x <_P y$ in $\phi_{\T}$ by $\phi_{SP}(x,y)$ 
(for any variables $x$ and $y$), we get a new sentence $\phi_{S\T}$. 
Then $(X, \prec_P, \prec_V) \models \phi_{S\T}$ if and only if $(X, \prec_{S(P)}, \prec_V) \models \phi_{\T}$, i.e., if and only if $(X, \prec_P, \prec_V) \in S^{-1}(\T)$.
This shows that $S^{-1}(\T)$ is definable in \TOTO.

The second statement of \cref{prop:S-sortable_definable} follows since $\I$ is definable in \TOTO by the sentence:
\[
\forall x \forall y \, \left( x <_P y \leftrightarrow x <_V y \right). \qedhere
\]
\end{proof}

\subsubsection{Sorting with a stack}

We start with the stack sorting operator \sortS, which has been considered most often in the permutation patterns literature.
This operator \sortS takes a permutation $\sigma$ and passes it through a stack, 
which is a \emph{Last In First Out} data structure. 
The contents of the stack are always in increasing order (read from top to bottom) i.e., an element can only be pushed onto the stack if it is less than the topmost element of the stack.
So, when a new element of the permutation is processed either it is pushed onto the stack if possible, or pops are made from the stack until it can be pushed (which might not be until the stack is empty). 
When no further input remains, the contents of the stack are flushed to output. The ordered sequence of output elements that results when a permutation $\sigma$ is processed is the permutation $\sortS(\sigma)$. 
For instance, for $\sigma \in \S_3$, $\sortS(\sigma) = 123$ except for $\sigma=231$. Indeed, $\sortS(231) = 213$ because 2 is pushed onto the stack and then popped to allow 3 to be pushed, 
then 1 is pushed, and then the stack is flushed.

\begin{proposition}
\label{prop:stackTOTOdef}
The stack sorting operator \sortS is \TOTO-definable. 
\end{proposition}

It follows immediately from \cref{prop:S-sortable_definable} that 
\begin{corollary}
\label{cor:stacksortable}
For any $k\geq 1$ the set of permutations sortable by $\sortS^k$ is described by a sentence in \TOTO. 
\end{corollary}

\begin{proof}[Proof of \cref{prop:stackTOTOdef}]
To see that $\sortS$ is \TOTO-definable we should consider, for $a$ and $b$ elements of some permutation $\sigma = (X, \prec_P, \prec_V)$, when $a \prec_{\sortS(P)} b$. 
\begin{itemize}
 \item If $a \prec_P b$ and $a \prec_V b$ then certainly $a \prec_{\sortS(P)} b$. Indeed, at the time $b$ is considered for pushing onto the stack, either $a$ will already have been popped, 
or $a$ will have to be popped to allow $b$ to be pushed. 
 \item If $a \prec_P b$ and $b \prec_V a$ then $a \prec_{\sortS(P)} b$ if and only if $a$ must be popped from the stack before $b$ is considered i.e., if and only if there is some $c$ with $a \prec_P c \prec_P b$ and $a \prec_V c$ (that is, $acb$ which occur in that positional order in $\sigma$, form the pattern $231$). 
 \item If $b \prec_P a$ then $a \prec_{\sortS(P)} b$ happens exactly when the two following conditions hold: $a \prec_V b$ and, for no $c$ positionally between $b$ and $a$ is it the case that $b \prec_V c$. 
\end{itemize}

Just from the language used in the descriptions above it should be clear that $\prec_{\sortS(P)}$ can be expressed by a \TOTO formula. For the record, its definition can be taken to be:
\[
x <_{\sortS(P)} y \, \Leftrightarrow \, 
\begin{array}{l} 
\ \left( x <_P y \et \left( x <_V y \ou \exists z \left( \, x <_P z <_P y \et x <_V z \right) \right) \right)  \\
\vee \left( y <_P x \et x <_V y \et \neg \exists z \left( \, y <_P z <_P x \et y <_V z \right) \right). \qedhere
\end{array} 
\]
\end{proof}

It is important to note that the proofs of Propositions~\ref{prop:S-sortable_definable} and~\ref{prop:stackTOTOdef} are constructive. 
Therefore, for any $k$, a \TOTO sentence whose models are exactly the $\sortS$-$k$-sortable permutations can be easily (and automatically) derived from those proofs. 
The formulas so obtained can be compared to the characterization of $\sortS$-$k$-sortable permutations 
by means of pattern-avoidance known in the literature \cite{knuth1968art,Ulfarsson:3stacks,west1993sorting}.
We explain below how to recover Knuth's result that $\sortS$-sortable permutations are those avoiding $231$. 
With a little more effort, West's characterization of $\sortS$-$2$-sortable permutations can be recovered in the same way. 
However, we were not able to ``read'' the characterization of $\sortS$-$3$-sortable permutations given by Ulfarsson on the obtained formula. 

By definition, \sortS fails to sort a permutation $\sigma = (X, \prec_P, \prec_V)$ if and only if there are $a,b \in X$ with $b \prec_{\sortS(P)} a$ but $a \prec_V b$. 
From the formula describing $x <_{\sortS(P)} y$ given at the end of the proof of \cref{prop:stackTOTOdef}, 
this is equivalent to the existence of $a$ and $b$ such that $a \prec_V b$, $b \prec_P a$, and there exists a $c$ with  $b \prec_P c \prec_P a$ and $b \prec_V c$. 
This says exactly that a permutation is stack-sortable if and only if it avoids the pattern $231$. So, we have recovered in a rather round-about way Knuth's original characterization of stack-sortability. 


\subsubsection{Other sorting operators}

We consider briefly two more sorting operators: bubble sort (\sortB), and sorting with a queue and bypass (\sortQ). 
We show that both are \TOTO-definable, hence from \cref{prop:S-sortable_definable}, that for any $k$, \sortB-$k$- and \sortQ-$k$-sortable permutations are described by a \TOTO sentence. 

\medskip

The bubble sort operator can be thought of as sorting with a one element buffer. When an element $a$ of the permutation is processed we:
\begin{itemize}
\item
place $a$ in the buffer if the buffer is empty (this only occurs for the first element of the permutation),
\item
output $a$ directly if the buffer is occupied by a value larger than $a$,
\item
output the element in the buffer and place $a$ in the buffer if the element in the buffer is smaller than $a$.
\end{itemize}
When the whole permutation has been processed we output the element remaining in the buffer. The result of sorting $\sigma$ with this mechanism is denoted $\sortB(\sigma)$. It is easy to see that:
\[
x <_{\sortB(P)} y \, \Leftrightarrow \, 
\begin{array}{l} 
\quad \left( x <_P y \et  \exists z \, (z \leq_P y \et x <_V z ) \right) \\
\ou \left( y <_P x \et \neg \exists z (z \leq_P x \et y <_V z) \right)
\end{array}
\]
The first clause captures the situation where $x$ precedes $y$ and is either output immediately (because the buffer is occupied by some larger element $z$ when $x$ is processed), 
or caused to be output by some such $z$ (possibly equal to $y$) up to the point when $y$ is processed. The second clause captures the only way that two elements can exchange positional order. 
The first must be able to enter the buffer (so there can be no larger preceding element) and must remain there up to and including the point at which the second arrives. 
As in the stack sorting case it is possible to use this definition to characterize the \sortB-sortable permutations as those that avoid the patterns 231 and 321.

\medskip

In sorting with a queue and bypass, we maintain a queue (First In First Out data structure) whose elements are in increasing order. When an element is processed it:
\begin{itemize}
\item
is added to the queue if it is greater than the last element of the queue,
\item
is output directly if it is less than the first element of the queue,
\item
causes all lesser elements of the queue to be output and is then added to the queue (if it is now empty), or output (if greater elements still remain in the queue).
\end{itemize}
It can easily be seen that being sortable by this algorithm
is equivalent to being sortable by two queues in parallel,
which is one of the models studied by Tarjan in \cite{tarjan1972sorting}.

Note that any element which is preceded (not necessarily immediately) by a greater element is output immediately (possibly after some elements have been removed from the queue), while any element which has no preceding greater element is added to the queue. For convenience let
\begin{align*}
\phi_{12} (x,y) &= x <_P y \et x <_V y, \\
\phi_{21} (x,y) &= x <_P y \et y <_V x ,\\
\phi_R(x) &= \neg \exists z \, \phi_{21}(z,x) , \\
\phi_{\text{out}} (x,y) &= \exists z, w \, (\phi_{21}(z,w) \et x \leq_P w \leq_P y ).
\end{align*}
The first three of these capture natural definitions ($\phi_R$ corresponding to being a left-to-right maximum). 
The fourth expresses the idea that ``$x$ didn't enter the queue or was removed from the queue by the arrival of some element $w$ before (or including) $y$''.
Then it is easy to see that:
\[
x <_{\sortQ(P)} y \, \Leftrightarrow \, 
\begin{array}{l} 
\phi_{12}(x,y) \ou \phi_{\text{out}}(x,y) \ou\\
\left( \phi_R(y) \et \neg \phi_{\text{out}}(y,x) \et \phi_{21}(y,x)  \right) . \\
\end{array}
\]


The first two clauses capture how $a$ precedes $b$ in $\sortQ(\sigma)$ if $a \prec_P b$. Namely, either $ab$ forms a $12$ pattern or $a$ never enters the queue at all, or it does but is caused to leave the queue by a subsequent addition to the queue (up to the arrival of $b$). The final clause covers the situation where $b \prec_P a$ but $a \prec_{\sortQ(P)} b$. For this to occur it must be the case that $b$ enters the queue and is not caused to leave before $a$ arrives (and then $a \prec_V b$ is also needed otherwise $a$ would either enter the queue as well or cause $b$ to be released from the queue).

Despite the apparent complexity of the definition of $\prec_{\sortQ(P)}$,
it is easy to infer that the \sortQ-sortable permutations are precisely those that avoid the pattern 321,
recovering a result in \cite{tarjan1972sorting}.
\medskip

We summarize our findings in the following proposition
\begin{proposition}
  \label{prop:Bubble_queue}
  The sorting operators $\sortS$ (stack sorting), $\sortB$ (bubble sort) and $\sortQ$ (queue and bypass)
  are all $\TOTO$-definable. As a consequence, for any composition $F$ of these sorting operators,
  the set $F^{-1}(\I)$ is definable in $\TOTO$.
\end{proposition}
\begin{remark}
There is a pleasant coincidence here: a permutation is \sortB-sortable if and only if it is both \sortS- and \sortQ-sortable. This is perhaps not unexpected since the only way to operate a queue or a stack without being able to determine which one is being used is to ensure that there is never more than one element stored (i.e., use it as a single element buffer), but we don't pretend to say that this should automatically follow on a logical basis!
\end{remark}

\subsection{A family of expressible cycle-types}\label{sec:expressibilityC_laD_la}

Although \TOTO is designed to express some pattern-related concepts, 
some information on the cycle decomposition of permutations is expressible in \TOTO. 
We start by an easy observation.
\begin{proposition}\label{prop:CLa_BiFO_TOTO}
For any partition $\la$, there exists a \TOTO formula expressing that a permutation is of cycle-type $\la$. 
\end{proposition}
\begin{proof}
There are a finite number of permutations of cycle-type $\la$. 
So we can simply take the disjunction, over all permutations $\sigma$ of cycle-type $\la$, 
of the \TOTO formula whose only model is $\sigma$ (see Proposition~\ref{prop:sameFO_samePerm}).
\end{proof}

Here is a more surprising result along these lines.
\revision{Recall that $\la \cup (1^k)$
is obtained by adding $k$ parts equal to $1$ to $\la$
(see the end of the introduction).}
\begin{proposition}\label{prop:DLa_BiFO_TOTO}
For any partition $\la$, there exists a \TOTO formula expressing that a permutation is of cycle-type $\la \cup (1^k)$, for some value of $k \geq 0$. 
\end{proposition}

\begin{proof}
As usual, for a partition $\la=(\la_1, \dots, \la_q)$, we denote by $|\la| = \sum_{i=1}^q \la_i$ its size. 
Being of cycle-type $\la \cup (1^k)$ for some value of $k$ can be translated as follows: there are $|\la|$ distinct elements which forms a pattern of cycle-type $\la$ and all other elements are fixed points. 
This is illustrated in \cref{fig:D32}, which shows a schematic representation of 
a permutation with cycle-type $(3,2) \cup (1^k)$.
\begin{figure}[t]
    \pgfmathsetmacro{\xa}{1+.8*rand}
    \pgfmathsetmacro{\xb}{\xa+1+.8*rand}
    \pgfmathsetmacro{\xc}{\xb+1+.8*rand}
    \pgfmathsetmacro{\xd}{\xc+1+.8*rand}
    \pgfmathsetmacro{\xe}{\xd+1+.8*rand}
    \pgfmathsetmacro{\xf}{\xe+1+.8*rand}
    \pgfmathsetmacro{\scalfac}{4/\xf}
    \def\esp{.1}
  \[\begin{tikzpicture}[scale=\scalfac]
    \draw (0,0) -- (0, \xf) -- (\xf, \xf) -- (\xf, 0) -- (0,0);
	\foreach \i in {\xa,\xb,\xc,\xd,\xe}{
		\draw[color=darkgray,dotted] (\i, 0) -- (\i, \xf);
		\draw[color=darkgray,dotted] (0, \i) -- (\xf, \i);
	}
    \absdot{(\xa,\xb)}{};
    \absdot{(\xb,\xd)}{};
    \absdot{(\xc,\xe)}{};
    \absdot{(\xd,\xa)}{};
    \absdot{(\xe,\xc)}{};
    \draw (\esp,\esp) -- (\xa - \esp,\xa - \esp);
    \draw (\xa+\esp,\xa+\esp) -- (\xb - \esp,\xb - \esp);
    \draw (\xb+\esp,\xb+\esp) -- (\xc - \esp,\xc - \esp);
    \draw (\xc+\esp,\xc+\esp) -- (\xd - \esp,\xd - \esp);
    \draw (\xd+\esp,\xd+\esp) -- (\xe - \esp,\xe - \esp);
    \draw (\xe+\esp,\xe+\esp) -- (\xf - \esp,\xf - \esp);
  \end{tikzpicture}\]
  \caption{A schematic representation of a permutation $\si$ with cycle-type
  $(3,2) \cup (1^k)$.
  The vertical and horizontal dotted lines indicate the positions and values
  of elements in the support of $\si$.
  The black points representing elements of the support
  form a pattern of cycle-type $(3,2)$.
  The diagonal lines represent an arbitrary number of points,
  placed in an increasing fashion, which are fixed points.}
  \label{fig:D32}
\end{figure}
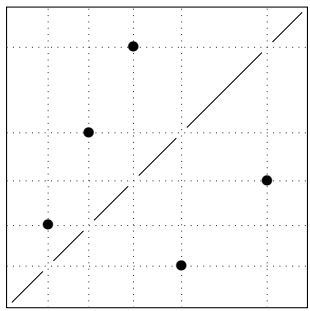

This does not immediately imply that being of cycle-type $\la \cup (1^k)$ is expressible in \TOTO, since \TOTO cannot \emph{a priori} express fixed points. 
However, this allows to translate the property of being of cycle-type $\la \cup (1^k)$ in a language that \TOTO can speak. 

Namely, for a permutation $\si$ and a subset $X=\{x_1,\dots,x_{|\la|}\}$ of size $|\la|$ of 
$A^\si=\{(i,\si(i))\, :\, 1\le i \le n\}$, 
we consider the following properties:
\begin{enumerate}[label=(P\arabic*)]
  \item the pattern $\pi$ induced by $\si$ on the set $X$ 
    has cycle-type $\la$;
    \label{item:Pattern_Type_la}
  \item the value and position orders coincide outside $X$;
    \label{item:Outside_Value=Position}
  \item for each element $y$ outside $X$ there are as many elements of $X$ below $y$ in the value order as in the position order.
    \label{item:Outside_Inside}
\end{enumerate}
We claim that for a permutation $\si$,
the existence of a set $X$ of size $|\la|$ with
properties \ref{item:Pattern_Type_la}, \ref{item:Outside_Value=Position}
and \ref{item:Outside_Inside} is equivalent to being of cycle-type $\la \cup (1^k)$. 
This will be proved later. 
For the moment, we focus on explaining why the existence of such a set $X$ with the above
three properties is indeed expressible in \TOTO. 
\medskip

Recall that \TOTO is a first-order theory,
so that we can only quantify on variables and not on sets.
However, since $X$ has a fixed size $|\lambda|$, quantifying on $X$
or on variables $x_1,\dots,x_{|\lambda|}$ is the same.
It is therefore sufficient to check that each of the properties
\ref{item:Pattern_Type_la}, \ref{item:Outside_Value=Position} and \ref{item:Outside_Inside}
are expressible in \TOTO by some \TOTO formula in the free variables $x_1,\dots,x_{|\la|}$.

For the first property,
we take as before the disjunction over all patterns $\pi$ with cycle-type $\la$. 
The second property writes naturally in terms of value and position orders
and is therefore trivially expressible in \TOTO.
For the third one, we have to be a bit more careful since sentences of the kind ``there are as many elements \dots as''
are in general not expressible in first-order logic.
But we are counting elements of $X$ with given properties and $X$ has a fixed size ($|X|=|\la|$).
So, we can rewrite \ref{item:Outside_Inside} as a finite disjunction over $j \le |\la|$ 
and over all pairs of subsets $X_P$ and $X_V$ of $X$ of size $j$ of properties
``the elements of $X$ below $y$ in the position order are exactly those of $X_P$ and 
the elements of $X$ below $y$ in the value order are exactly those of $X_V$''. 
This shows that the existence of a set $X$ of size $|\la|$ such that properties \ref{item:Pattern_Type_la}, \ref{item:Outside_Value=Position}     
and \ref{item:Outside_Inside} hold is expressible in \TOTO. 

\medskip

We now prove our claim that, for a permutation $\si$,
the existence of a set $X$ of size $|\la|$ with
properties \ref{item:Pattern_Type_la}, \ref{item:Outside_Value=Position}
and \ref{item:Outside_Inside} is equivalent to being of cycle-type $\la \cup (1^k)$, for any value of $k \geq 0$. 

\medskip

First take a permutation $\si$ of cycle-type $\la \cup (1^k)$.
We consider its support (i.e. the set of non fixed points), denoted $I_{s}$. 
And if $\la$ has $m_1$ parts equal to $1$, we choose arbitrarily a set $I_f$ of $m_1$ fixed points of $\si$.
This gives a set $I = I_{s} \uplus I_f$ such that $\si(I)=I$ and such that the restriction of $\si$ to $I$
has cycle-type $\la$. We set $X=\{(i,\si(i)),\, i \in I\}$.
In particular, $\si$ induces a pattern $\pi$ of cycle-type $\la$ on $X$,
so that \ref{item:Pattern_Type_la} is satisfied.
Property \ref{item:Outside_Value=Position} is also satisfied since all points
outside $X$ are fixed points for the permutation $\si$.
Consider property \ref{item:Outside_Inside}.
We take $x=(i,\si(i))$ in $X$ and $y$ outside $X$.
Since $y$ is a fixed point i.e. $y=(j,j)$ for some $j$, 
the relation $x <_V y$ is equivalent to $\si(i) < j$.
Besides, by definition, $x <_P y$ means $i<j$.
But $\si(I)=I$ so that there are as many $i$ in $I$ such that $i < j$ as $i$ in $I$ such that $\si(i) < j$.
This proves that $X$ also satisfies property \ref{item:Outside_Inside}.
\medskip

Conversely, we consider a permutation $\si$ and we assume that there is a subset $X$
of elements of $\si$ satisfying \ref{item:Pattern_Type_la}, \ref{item:Outside_Value=Position}
and \ref{item:Outside_Inside}. We want to prove that $\si$ has cycle-type $\la \cup (1^k)$. 

Let $I$ be such that $X=\{(i,\si(i)),\, i \in I\}$.
We first prove that $\si(I)=I$.
Consider the smallest element $x_1=(i_1,\si(i_1))$ of $X$ in the position order
and its smallest element $y_1=(j_1,\si(j_1))$ in the value order.
By property \ref{item:Outside_Inside}, an element $y$ outside $X$
is below $x_1$ in the position order if and only if it is below $y_1$ in the value order.
Therefore the rank of $x_1$ in the position order (which is $i_1$)
is the same as the rank of $y_1$ in the value order (which is $\si(j_1)$),
so that $\si(j_1)=i_1$.
Similarly, if $x_k=(i_k,\si(i_k))$ (resp. $y_k=(j_k,\si(j_k))$) is the $k$-th smallest element of $X$ in the position (resp. value) order,
then $\si(j_k)=i_k$, which proves $\si(I)=I$.

We now claim that elements $y$ outside $X$ are fixed points for $\si$.
To do that, we should prove that there are as many elements $z$ below $y$ in the position order as in the value order.
Property \ref{item:Outside_Inside} says that this holds when restricting to elements $z$ of $X$.
But property \ref{item:Outside_Value=Position} tells us that an element $z$ in the complement $X^c$ of $X$ is below $y$ in the position order
if and only if it is below $y$ in the value order, so that there is the same number of elements $z$ of $X^c$ below $y$ in both orders.
We can thus conclude that $y$ is a fixed point for $\si$.

By \ref{item:Pattern_Type_la}, the pattern induced by $\si$ on $I$ has cycle-type $\la$.
Moreover, since $\si(I)=I$, the restriction of $\si$ to $I$ is a union of cycles of $\si$, 
so that $\si$ has cycle-type $\lambda \cup \mu$ for some $\mu$. 
On the other hand, we have proved that elements outside $I$ are fixed points for $\si$.
Therefore $\mu = (1^{|X^c|})$ and $\si$ has cycle-type $\la \cup (1^k)$, as wanted.
\end{proof}

\section{Proving inexpressibility results: Ehrenfeucht-Fra\"{i}ss\'{e} games on permutations}
\label{sec:EFGames}

As we have seen above with many examples, proving that a property of permutations is expressible in \TOTO is easy, at least in principle: 
it is enough to provide a \TOTO formula expressing it. 
But how to prove that a property is \emph{not} expressible in \TOTO? 
To this end, we present a technique -- that of Ehrenfeucht-Fra\"{i}ss\'{e} games -- to show inexpressibility results in \TOTO. 
This method is very classical, although its application in the context of permutations seems to be new.

\subsection{The theory of Ehrenfeucht-Fra\"{i}ss\'{e} games}
Ehrenfeucht-Fra\"{i}ss\'{e}, or Duplicator-Spoiler games are a fundamental tool in proving definability and non-definability results. 
We give a brief introduction to them in the context of the permutations as models of \TOTO below, 
but also refer the reader to \cite{Gradel2007Finite-model-th,Hodges1997A-shorter-model,Immerman1999Descriptive-com,Rosenstein1982Linear-ordering} for various presentations with differing emphases.

Let $\alpha$ and $\beta$ be two permutations (or more generally any two models of the same theory), and let $k$ be a positive integer. The \emph{Ehrenfeucht-Fra\"{i}ss\'{e} (EF) game of length $k$} played on $\alpha$ and $\beta$ is a game between two players (named \emph{Duplicator} and \emph{Spoiler}) according to the following rules:
\begin{itemize}
\item
The players alternate turns, and Spoiler moves first.
\item
The game ends when each player has had $k$ turns.
\item
At his $i\th$ turn, ($1 \leq i \leq k$) Spoiler chooses either an element $a_i \in \alpha$ or an element $b_i \in \beta$. In response, at her $i\th$ turn, Duplicator chooses an element of the other permutation.
Namely, if Spoiler has chosen $a_i \in \alpha$, then Duplicator chooses an element $b_i \in \beta$, 
and if Spoiler has chosen $b_i \in \beta$, then Duplicator chooses $a_i \in \alpha$.
\item
At the end of the game if the map $a_i \mapsto b_i$ for all $i \le k$
preserves both position and value orders, then Duplicator wins
(more generally, she wins if the map $a_i \mapsto b_i$ defines an isomorphism between
the submodels of $\alpha$ and $\beta$ consisting of $\{a_i, 1 \le i \le k\}$ and $\{b_i, 1 \le i \le k\}$,
respectively). Otherwise, Spoiler wins.
\end{itemize}

We assume that the players play in the best possible way, i.e.
we say that Duplicator {\em wins} if she has a winning strategy.
Put briefly for the case of permutations, Duplicator wins if she can at every turn, choose a point which corresponds to the point just chosen by Spoiler in the other permutation, 
in the sense that it compares in the same way to all previously chosen points,
both for the value and the position orders. If Duplicator wins the $k$-move EF game on $\alpha$ and $\beta$ we write $\alpha \EFeq_k \beta$. It is easy to check that $\EFeq_k$ is an equivalence relation for each $k$.

Based on the recursive definitions of formulas we can define the \emph{quantifier depth}, $\qdepth(\psi)$, of a formula $\psi$ in an obvious way. If $\psi$ is an atomic formula then $\qdepth(\psi) = 0$. Otherwise:
\begin{align*}
\qdepth(\neg \psi) &= \qdepth (\psi), \\
\qdepth(\psi \vee \theta) = \qdepth(\psi \wedge \theta) &= \max (\qdepth(\psi), \qdepth(\theta)), \\
\qdepth(\exists x \, \psi) = \qdepth(\forall x \, \psi) &= \qdepth(\psi) + 1. 
\end{align*}

The connection between EF games and quantifier depth is captured in the fundamental theorem of Ehrenfeucht and Fra\"{i}ss\'{e}: 
\begin{quote}
$\alpha \EFeq_k \beta$ if and only if $\alpha$ and $\beta$ satisfy the same set of sentences of quantifier depth $k$ or less.
\end{quote}
One way to use this result is to establish when two models satisfy the same set of sentences (of all quantifier depths). 
However, this is not so interesting in the finite context where this is equivalent to isomorphism. 
Instead, we are interested in its application in proving \emph{inexpressibility} results. 
Consider some property, $P$, of permutations (or of models over some other signature). 
We would like to know if $P$ is \emph{expressible in \TOTO}, 
that is whether or not there is some \TOTO sentence $\psi$ such that $\alpha$ is a model of $\psi$ if and only if $\alpha$ has property $P$. 
To show that this is \emph{impossible} it suffices to demonstrate, for each positive integer $k$, that 
there are permutations (resp.~models) $\alpha$ and $\beta$ such that Duplicator wins the $k$-move EF game on $\alpha$ and $\beta$, 
and $\alpha$ satisfies $P$ but $\beta$ does not. 

We begin with a simple but illustrative example which uses a well-known result about EF games for finite linear orders (see, e.g., Theorem 2.3.20 of \cite{Gradel2007Finite-model-th}). We include the proofs as they capture the flavour of what is to follow.

\begin{proposition}
  \label{prop:EF_linear_orders}
Let a positive integer $k$ be given. If $X$ and $Y$ are finite linear orders and $|X|, |Y| \geq 2^k - 1$, then Duplicator wins the $k$-move EF game played on $(X,Y)$.
\end{proposition}

\begin{proof}
After the $r$-th move, $r$ elements have been chosen in $X$ and this defines $r+1$ intervals $I_0,\dots,I_r$,
where $I_i$ \revision{consists of the elements of $X$ 
that are greater than exactly $i$ out of the $r$ elements chosen so far.}
Similarly, elements chosen in $Y$ define $r+1$ intervals $J_0$,\dots,$J_r$.

The basic idea of the proof is simple. 
For a fixed integer parameter $a$ call a subinterval of a finite linear order consisting of at least $2^a-1$ elements \emph{$a$-long}.
The following claim is easily proved by iteration on $r$ (from $0$ to $k$): 
regardless of Spoilers' moves, Duplicator can always arrange that after $r$ moves,
for any $s$ in $\{0,\dots,r\}$ either $I_s$ and $J_s$ are both $(k-r)$-long, or have exactly the same length. 

Applying this strategy on the $k$-move EF game played on $(X,Y)$,
the suborders of $X$ and $Y$ induced by the chosen elements are clearly isomorphic, 
proving our proposition. 
\end{proof}

\begin{corollary}
There is no first-order sentence to express that a finite linear order contains an even number of elements. 
\end{corollary}

\begin{proof}
The proof is by contradiction. 
Assume such a sentence $\phi$ would exist. It would have a specific quantifier depth, $k$. 
Consider $X$ and $Y$ two linear orders with $|X|=2^k$ and $|Y|=2^k-1$. Of course, $X \models \phi$ while $Y \not \models \phi$.
But from Proposition~\ref{prop:EF_linear_orders}, we have $X \EFeq_k Y$, which brings a contradiction to 
the fundamental theorem of Ehrenfeucht and Fra\"{i}ss\'{e}. 
\end{proof}

\subsection{Some consequences for permutations}

From now on, for any integer $n$, let us denote by $\inc_n = 12 \dots n$ (resp. $\delta_n =n \dots 21$) the increasing (resp. decreasing) permutation of size $n$. 

\begin{proposition}
  \label{Prop:monotone_equivalent} 
  Let $n$, $m$ and $k$ be positive integers with $n,m \ge 2^k-1$.
  Then we have $\inc_m \sim_k \inc_n$ and $\delta_m \sim_k \delta_n$.
\end{proposition} 

\begin{proof}
  In an increasing permutation,
  position and value order coincide.
  Therefore the EF game on $\inc_m$ and $\inc_n$ can be played as 
  if it was played on linear orders of $m$ and $n$ elements respectively.
  We conclude by \cref{prop:EF_linear_orders}
  that Duplicator wins the $k$-move Duplicator-Spoiler game.

  The same is true for decreasing permutations $\delta_m$ and $\delta_n$
  since in this case, the value order is just the opposite of the position order.
\end{proof}

\begin{corollary}
  \label{Cor:Fixed_Points_Not_FO}
The property of having a fixed point is not expressible by a sentence in \TOTO. 
In other words, there does not exist a sentence $\psi$ in \TOTO such that $\sigma \models \psi$ if and only if $\sigma$ has a fixed point.
\end{corollary}

\begin{proof}
Consider the monotone decreasing permutation $\delta_n$ on $n$ elements. If $n$ is odd, then $\delta_n$ has a fixed point, while if it is even, it does not. 
Suppose that a sentence $\psi$ of $\TOTO$ defined ``having a fixed point'' and had quantifier depth $k$.
Then, from \cref{Prop:monotone_equivalent}, either both $\delta_{2^k-1}$ and $\delta_{2^k}$ would satisfy $\psi$ or neither would.
However, one has a fixed point and the other does not, which contradicts the supposed property of $\psi$.
\end{proof}

%
%
%
%
%

We will use the underlying idea of the previous argument in a number of different contexts in what follows. 
Specifically, we will generally construct permutations $\alpha$ and $\beta$ which are similar enough that Duplicator wins the EF game of a certain length, 
but are different enough that one has the property under consideration while the other does not. 
Frequently this `similarity' will involve some embedded monotone subsequences, 
and we will use notation as inflations, sums and skew-sums to describe such situations in a uniform way.

The following proposition is the key ingredient in most of the arguments in the next section.

\begin{proposition}
\label{Prop:MagicLemma}
Let $\alpha \in \S_n$ and for $1 \leq i \leq n$ suppose that $\sigma_i \EFeq_k \tau_i$. Then $\alpha[\sigma_1, \sigma_2, \dots, \sigma_n] \EFeq_k \alpha[\tau_1, \tau_2, \dots, \tau_n]$.
\end{proposition}

\begin{proof}
It is easy to demonstrate a winning strategy for Duplicator in the EF game of length $k$ on $\alpha[\sigma_1, \sigma_2, \dots, \sigma_n]$ and $\alpha[\tau_1, \tau_2, \dots, \tau_n]$. She simply keeps track of $n$ EF games, one in each of the pairs $\sigma_i$ and $\tau_i$. Whenever Spoiler makes a move, she notes the corresponding $\sigma_i$ or $\tau_i$, and responds in that game (and the corresponding part of $\alpha$), according to the winning strategy guaranteed by $\sigma_i \EFeq_k \tau_i$.
Since the relationships between different $\sigma$s and $\tau$s are fixed, her move maintains an isomorphism, and since she is never required to move more than $k$ times in any particular $\sigma_i$ or $\tau_i$ she wins at the end.
\end{proof}

There is a corresponding result when 
the inflated permutations are only related by $\EFeq$,
and when we inflate all points by the {\em same} permutation.

\begin{proposition}
Let $\alpha \in \S_n$ and $\beta \in \S_m$ and suppose that $\alpha \EFeq_k \beta$.
Take an arbitrary permutation $\sigma$.
Then $\alpha[\sigma, \sigma, \dots, \sigma] \EFeq_k \beta[\sigma, \sigma, \dots, \sigma]$.
\end{proposition}

\begin{proof}
\revision{By construction,
the elements of $\alpha[\sigma, \sigma, \dots, \sigma]$
are partitioned into $|\alpha|$ copies of the diagram of $\sigma$,
each copy corresponding to one element $(i,\alpha(i))$ 
(or $\alpha(i)$ for short) of the diagram of $\alpha$.
In this proof, we will refer to the copy corresponding to $\alpha(i)$
as the {\em $\sigma$-inflation} of $\alpha(i)$.
We use a similar terminology for $\beta[\sigma, \sigma, \dots, \sigma]$.}

This time Duplicator needs to follow the winning strategy which is guaranteed from $\alpha \EFeq_k \beta$.
Whenever Spoiler makes a move, 
\revision{Duplicator first notes the element $\alpha(i)$ or $\beta(i)$ 
such that Spoiler's move has been done in the $\sigma$-inflation of $\alpha(i)$ or $\beta(i)$; 
Duplicator then responds in the $\sigma$-inflation of the $\beta(j)$ or $\alpha(j)$ given by the winning strategy for the game on $(\alpha,\beta)$.
Within this copy of $\sigma$, Duplicator chooses the same
element as Spoiler has chosen in his copy.}
%

\revision{
We note that at the end of the game,
there can be several marked elements in the same copy of $\sigma$.
That relations between elements in different copies of $\sigma$
are the same 
in $\alpha[\sigma, \sigma, \dots, \sigma]$ and $\beta[\sigma, \sigma, \dots, \sigma]$
is ensured by Duplicator following
a winning strategy for the game on $(\alpha,\beta)$.
Since relations inside a single copy of $\sigma$ are the same on both sides,
Duplicator's moves always maintain an isomorphism and she will triumph.}
\end{proof}


\subsection{EF game for formulas with free variables}
In the sequel we will also want to prove results about the (in)expressibility of certain properties 
of \emph{elements of} permutations (as opposed to properties of the permutations themselves). 
Recall that expressing such a property in $\TOTO$ means representing it by a formula $\phi(\mathbf{x})$ having one or more free variables, $\mathbf{x}$. 
And for a sequence $\mathbf{a}$ of elements from a permutation $\sigma$, the property would be satisfied by $\mathbf{a}$ in $\sigma$ if and only if $(\pi,\mathbf{a}) \models \phi(\mathbf{x})$.

%
%
There is a standard modification of EF games that allows one to demonstrate inexpressibility in this case as well. 
Similarly to the previous case, for two permutations $\alpha$ and $\beta$, each equipped with a sequence of ``marked'' elements $\mathbf{a}$ and $\mathbf{b}$ of the same size, say $r$,  
we write $(\alpha, \mathbf{a}) \EFeq_k (\beta, \mathbf{b})$ if both pairs satisfy the same formulas with $r$ free variables and quantifier-depth at most $k$.
Writing $\mathbf{a}=(a_1,\dots,a_r)$ and $\mathbf{b}=(b_1,\dots,b_r)$, 
the modified EF game with $k$ rounds on such a pair $(\alpha, \mathbf{a})$ and $(\beta, \mathbf{b})$ goes as follows: 
for each $i \leq r$, at round $i$, Spoiler must choose $a_i$ and Duplicator $b_i$; 
then, $k$ additional rounds are played, as in a classical EF game. 
The usual isomorphism criterion (on sequences of $r+k$ chosen elements) determines the winner. 
%
It can be proved that $(\alpha, \mathbf{a}) \EFeq_k (\beta, \mathbf{b})$ if and only if
Duplicator has a winning strategy in this game.

We illustrate this method of proof with Proposition~\ref{prop:Fixed_Points_Not_FO_formula} below. 

\begin{proposition}
The property that a given element of a given permutation is a fixed point is not expressible by a formula in \TOTO. 
In other words, there does not exist a formula $\phi(x)$ such that $(\sigma,a) \models \phi(x)$ if and only if $a$ is a fixed point of $\sigma$.
\label{prop:Fixed_Points_Not_FO_formula}
\end{proposition}

Of course, it should be noticed that Proposition~\ref{prop:Fixed_Points_Not_FO_formula} is also an immediate consequence of Corollary~\ref{Cor:Fixed_Points_Not_FO}. 
Indeed, if such a formula $\phi(x)$ were to exist, quantifying existentially over $x$ would provide a sentence (namely, $\exists x \,  \phi(x)$)
expressing the existence of a fixed point in \TOTO, therefore contradicting \cref{Cor:Fixed_Points_Not_FO}.

\begin{proof}

Recall that $\inc_n$ denotes the monotone increasing permutation on $n$ elements, and let $\pi_{m, n} = \inc_m \ominus 1 \ominus \inc_n = 321[\inc_m, 1, \inc_n]$. 
Denote by $a_{m, n}$ the ``central'' element of $\pi_{m, n}$, that is to say the one which inflates the $2$ in $321$. 
Clearly, $a_{m, n}$ is a fixed point of $\pi_{m,n}$ if and only if $m = n$. 

Now, assume that there exists a formula $\phi(x)$ expressing the property that $x$ is a fixed point, 
and denote by $k$ its quantifier depth. 
Taking $n=2^k-1$, it holds that $(\pi_{n, n},a_{n, n}) \EFeq_k (\pi_{n, n+1},a_{n, n+1})$. 
This follows indeed from Propositions~\ref{Prop:monotone_equivalent} and~\ref{Prop:MagicLemma}. 
On the other hand, $a_{n, n}$ is a fixed point of $\pi_{n, n}$ whereas  $a_{n, n+1}$ is not a fixed point of $\pi_{n, n+1}$, 
bringing a contradiction to the fundamental theorem of Ehrenfeucht and Fra\"{i}ss\'{e}. 
\end{proof}

\section{Expressivity of restrictions of \TOTO to permutation classes}
\label{Sect:Definability_Cycles}

As we discussed in the introduction, \TOTO is not designed to express properties related to the cycle structure of permutations. 
And indeed, in general, \TOTO cannot express such properties. 
We have seen already with \cref{Cor:Fixed_Points_Not_FO} and \cref{prop:Fixed_Points_Not_FO_formula} that
the simplest statements of ``having a fixed point'' or ``a given element is a fixed point'' are not expressible in \TOTO. 
We will see later with \cref{thm:ExpressCycles} that \TOTO is also unable to express that a sequence of elements forms a cycle. 

In this section, we consider restrictions of \TOTO and ask whether some properties of the cycle structure of permutations 
(like containing a fixed point or a cycle of a given size) become expressible in such restricted theories. 
We focus on the restricted theories \TOTOC: 
the signature is the same as in \TOTO, 
but \TOTOC has additional axioms, to ensure that the models considered are only the permutations belonging to some permutation class $\C$. 
Recall that a permutation class is a set of permutations that is closed downward by extraction of patterns, 
and that every permutation class can be characterized by the avoidance of a (possibly infinite) set of classical patterns
\cite{Bevan:brief}.

Our goal is a complete characterization of the permutation classes $\C$ for which \TOTOC can express the fixed point (resp. longer cycle) property. 

We actually consider two versions of this problem. 
With the ``sentence'' version, we are interested in finding (or proving the existence of) sentences expressing the existence of a fixed point (resp. of a certain cycle). 
For the ``formula'' version of the problem, given an element (resp. a sequence of elements) of a permutation, 
we ask whether there is a formula with free variable(s) expressing the property that this element is a fixed point (resp. this sequence of elements is a cycle). 

Of course, a positive answer to the ``formula'' problem implies a positive answer to the ``sentence'' problem, 
simply by quantifying existentially over all free variables of the formula. But \emph{a priori} the converse need not be true. 
{\em A posteriori}, we will see that in the case of fixed point the converse does hold
(see Theorem~\ref{thm:fixed_point_expressible_in_classes}):
there exists a sentence expressing the existence of fixed points in \TOTOC
if and only if there exists a formula expressing that a given element is a fixed point in \TOTOC.
This does not however generalize to larger cycles.

\subsection{Fixed points}
In the simplest case of fixed points, 
we have seen with Corollary~\ref{Cor:Fixed_Points_Not_FO} and Proposition~\ref{prop:Fixed_Points_Not_FO_formula}
that there is neither a formula nor a sentence expressing the property of being/the existence of a fixed point in the unrestricted theory \TOTO.

Taking a closer look at the proofs of these statements, we actually know already of some necessary conditions for a class $\C$ 
to be such that fixed points are expressible in \TOTOC. 
First, the proof of \cref{Cor:Fixed_Points_Not_FO} shows that if $\C$ contains all the permutations $\delta_m$ 
then the existence of a fixed point is not expressible in the models belonging to $\C$.
Put formally, we have the following.
\begin{lemma}
\label{lem:NecCond1}
If there exists a sentence in \TOTOC expressing the existence of a fixed point, then $\C$ must avoid at least one decreasing permutation $\delta_k$.
\end{lemma}

A similar statement along these lines is the following.
\begin{lemma}
\label{lem:NecCond2}
If there exists a sentence in \TOTOC expressing the existence of a fixed point,
then $\C$ must avoid at least one permutation of the form $\pi_{m, n} = 321[\inc_m, 1, \inc_n]$.
\end{lemma}

\begin{proof}
The essential argument is in the proof of Proposition~\ref{prop:Fixed_Points_Not_FO_formula}. 
Assume that there exists such a sentence, of quantifier depth $k$. 
Assume in addition that all permutations $\pi_{m, n}$ belong to \C.
Let $n = 2^k-1$, and remark that $\pi_{n, n}$ and $\pi_{n, n+1}$ both belong to $\C$. 
Combining \cref{Prop:MagicLemma,Prop:monotone_equivalent}, we get  $\pi_{n, n} \EFeq_k \pi_{n, n+1}$. 
On the other hand,
remark that $\pi_{n, n}$ has a fixed point while $\pi_{n, n+1}$ does not, 
bringing a contradiction to the fundamental theorem of Ehrenfeucht and Fra\"{i}ss\'{e}. 
\end{proof}
We can conclude that if $\C$ permits the definition of fixed points in the sentence sense,
then $\C$ must not contain all the permutations $321[\inc_m, 1, \inc_n]$, 
in addition to not containing all decreasing permutations. In fact, we shall prove that these conditions are sufficient. 

\begin{theorem}
\label{thm:fixed_point_expressible_in_classes}
  Let $\C$ be a permutation class. The following are equivalent.
  \begin{enumerate}
    \item There exists a sentence $\psi \in \TOTOC$ that expresses the existence of fixed points in $\C$;
      namely, for $\sigma$ in $\C$, $\sigma \models \psi$ if and only if $\sigma$ has a fixed point.
    \item There exists a formula $\phi(x) \in \TOTOC$ that expresses the fact that $a$ is a fixed point of $\sigma$;
      namely, for $\sigma$ in $\C$ and $a$ an element of $\sigma$,
      $(\sigma,a) \models \phi(x)$ if and only if $a$ is a fixed point of $\sigma$.
    \item There exist positive integers $k$, $m$ and $n$ such that $\delta_k$ and $321[\inc_m, 1, \inc_n]$
      do {\em not} belong to the class $\C$.
  \end{enumerate}
\end{theorem}
\revision{Before starting the proof, we recall for the reader's convenience
a famous result of Erd\H{o}s and Szekeres \cite{Erdos1935A-combinatorial},
which is used in several proofs throughout this section:
any permutation of size at least $ab+1$ contains either 
an increasing subsequence of size $a+1$ or a decreasing one of size $b+1$.
}
\begin{proof}
As mentioned earlier, 2 trivially implies 1 (by existential quantification over $x$). 
We have also seen in Lemmas~\ref{lem:NecCond1} and~~\ref{lem:NecCond2} that, 
if there exists a sentence expressing the existence of a fixed point in \C, 
then \C must exclude one permutation $\delta_k$ and one permutation of the form $321[\inc_m, 1, \inc_n]$; 
in other words 1 implies 3.
Finally, we prove that 3 implies 2.
Suppose that there exist positive integers $k$, $m$ and $n$ such that $\delta_k \notin \C$ and $321[\inc_m, 1, \inc_n]\notin \C$.
\revision{Let $\sigma \in \C$ be given and consider a point $a$
in the domain $A^\sigma$ of $\sigma$
(recall from \cref{sec:perm_as_models} that, in \TOTO,
$A^\sigma=\{(i,\sigma(i)),\, 1\le i \le n\}$).}
We consider
\begin{align*}
  U &= \left\{ y \in \revision{A^\sigma} : y <_P a \: \mbox{and} \: y >_V a \right\} \\
  V &= \left\{ y \in \revision{A^\sigma} : y >_P a \: \mbox{and} \: y <_V a \right\} .
\end{align*}
Clearly, $a$ is a fixed point if and only if $|U| = |V|$.

Assume now that $a$ is a fixed point.
We will show that there is a bound on $|U|$, uniform on all fixed points of all permutations $\sigma$ in $\C$
(but depending on $\C$).
Suppose without loss of generality that $m \geq n$.
If $|U| = |V| > (m-1)(k-3)$ then the subpermutations of $\sigma$ on $U$ and $V$ must both contain either $\delta_{k-2}$ or $\inc_m$
by the Erd\H{o}s-Szekeres theorem. 
But if either contained $\delta_{k-2}$ then $\sigma$ would contain $\delta_k$, 
while if both contained $\inc_m$ then $\sigma$ would contain $321[\inc_m, 1, \inc_n]$.
Since $\C$ is a class, contains $\sigma$ but contains neither $\delta_k$ nor $321[\inc_m, 1, \inc_n]$,
in all cases we reach a contradiction.

So, if $a$ is a fixed point then necessarily $|U|, |V| \leq (m-1)(k-3)$.
Now, using the above characterization of fixed points, 
it is easy to see how to construct a formula (with one free variable $x$) that expresses that $a$ is a fixed point: 
namely, we take the (finite) disjunction over $i \leq (m-1)(k-3)$ of
``there exist exactly $i$ points in $U$ and exactly $i$ points in $V$''.
\end{proof}

\begin{remark}
  The following fact was used implicitly in the above proof, and will also be in subsequent proofs.
  Let $\phi_1(x)$ and $\phi_2(x)$ be formulas in \TOTO.
  For a fixed $k$, the property ``there exists exactly $k$ elements $x$ satisfying $\phi_1$''
  is expressible in \TOTO (and similarly with $\phi_2$). Simply write
  \begin{align*}
    \exists y_1 \dots \exists y_{\revision{k}} \, & \big(y_1 \neq y_2 \wedge \dots \wedge y_1 \neq y_{\revision{k}} \wedge \dots \wedge y_{{\revision{k}}-1} \neq y_{\revision{k}}\big) \\
  & \wedge \big(\phi_1(x) \leftrightarrow (x=y_1 \vee \dots \vee x=y_k) \big) 
\end{align*}
It is however not in full generality possible to express the fact that there
are as many elements satisfying $\phi_1$ as elements satisfying $\phi_2$.
Therefore, being able to bound the number of elements in $U$ and $V$ in the above proof is key.
\end{remark}
\begin{remark}
\label{rk:fixed_point_expressible_in_classes}
In preparation of the next section, it is helpful to notice that, 
in the proof of \cref{thm:fixed_point_expressible_in_classes}, 
the Erd\H{o}s-Szekeres theorem can also be used to identify 
the obstructions $\delta_k$ and $321[\inc_m, 1, \inc_n]$. 
Indeed, consider \revision{$\sigma \in \C$, $a \in A^\sigma$} and suppose that $a$ is a fixed point of $\sigma$. 
Consider $U$ and $V$ as defined in the above proof, and assume that $|U|$ and $|V|$ are ``large'', namely at least $(m-1)^2+1$. Then,
using the Erd\H{o}s-Szekeres theorem, in each of those two regions we can choose a monotone subsequence of length $m$, 
and deleting all elements other than $a$ and those subsequences we see that $\C$ contains at least one of the following four permutations:
\begin{center}
	\begin{tikzpicture}[scale=0.5]
		\begin{scope}[shift={(0,0)}]
			\draw (0,0) grid (2,2);
			\draw (0.1, 1.9) -- (0.9, 1.1);
			\draw (1.1, 0.9) -- (1.9, 0.1);
			\draw (1.0, -0.4) node {\scriptsize $\alpha_{1,m}$};
			\fill (1,1) circle (0.1);
		\end{scope}
		\begin{scope}[shift={(3,0)}]
			\draw (0,0) grid (2,2);
			\draw (0.1, 1.1) -- (0.9, 1.9);
			\draw (1.1, 0.9) -- (1.9, 0.1);
			\draw (1.0, -0.4) node {\scriptsize $\alpha_{2,m}$};
			\fill (1,1) circle (0.1);
		\end{scope}
		\begin{scope}[shift={(6,0)}]
			\draw (0,0) grid (2,2);
			\draw (0.1, 1.9) -- (0.9, 1.1);
			\draw (1.1, 0.1) -- (1.9, 0.9);
			\draw (1.0, -0.4) node {\scriptsize $\alpha_{3,m}$};
			\fill (1,1) circle (0.1);
		\end{scope}
		\begin{scope}[shift={(9,0)}]
			\draw (0,0) grid (2,2);
			\draw (0.1, 1.1) -- (0.9, 1.9);
			\draw (1.1, 0.1) -- (1.9, 0.9);
			\draw (1.0, -0.4) node {\scriptsize $\alpha_{4,m}$};
			\fill (1,1) circle (0.1);
		\end{scope}
	\end{tikzpicture}
\end{center}
where in each case both the monotone segments contain $m$ points. 
Since \C is a class, it holds that if \C contains $\alpha_{i,m}$, then \C also contains all the permutations of the same general shape but with the two monotone sequences of arbitrary length at most $m$.
Assuming that there is no bound on $|U|$ and $|V|$, the usual EF game argument based on \cref{Prop:monotone_equivalent} 
shows that there is no formula expressing the property of being a fixed point in \TOTOC  
(as done in the proof of Lemma~\ref{lem:NecCond2}).
So, for fixed points to be definable, for each $i$, $\C$ must avoid some $\alpha_{i,m}$. 
Note though that avoiding $\alpha_{1,m}$ already implies avoiding $\alpha_{2,2m+1}$ and $\alpha_{3,2m+1}$ so we can ignore the middle two configurations 
thus obtaining our necessary obstructions $\delta_k$ and $321[\inc_m, 1, \inc_n]$, which can next be proved to be sufficient for expressibility of fixed points. 
\end{remark}

\subsection{Larger cycles and stable subpermutations}

After describing the classes $\C$ where fixed points are expressible by a formula in \TOTOC, 
we wish to similarly describe classes in which we can express that a given sequence of elements is a cycle. 
This is achieved with \cref{thm:ExpressCycles} below, 
the notation $E(\pi, \mathbf{i}, \overline{\X})$ used in this theorem being defined later in this section. 
The classes $\I$ and $\D$ are the classes of monotone increasing and decreasing permutations, respectively. 

\begin{theorem}
\label{thm:ExpressCycles}
Let $\C$ be a permutation class, and $k$ be an integer. The following are equivalent.
\begin{enumerate}
\item There is a formula $\phi(\mathbf{x})$ of \TOTOC with free variables $\mathbf{x} = (x_1, \dots, x_k)$ 
such that for all $\sigma \in \C$ and all sequences $\mathbf{a} = (a_1, \dots, a_k)$ of elements of $\sigma$, 
$(\sigma,\mathbf{a}) \models \phi(\mathbf{x})$ if and only if $\mathbf{a}$ is a cycle of $\sigma$.
\item For each $k$-cycle $\pi$, for each non-trivial cycle $\mathbf{i}$ of distinct elements from $[k+1]$, 
and for each sequence $\overline{\X}$ of the same length as $\mathbf{i}$, consisting of $\I$s and $\D$s, 
the class $\C$ avoids at least one permutation in each class of the form $E(\pi, \mathbf{i}, \overline{\X})$.
\end{enumerate}
\end{theorem}

Several remarks about this theorem should be made. 
First, in the case of transpositions, the second condition takes a neater form,
similar to the one we found for fixed points.
Next, note that it characterizes the classes $\C$ for which there exists a formula (with free variables) expressing that a given sequence of elements is a cycle, 
but it does not characterize the classes where the \emph{existence} of a cycle could be expressed by a sentence in \TOTOC. 
We have not been able to provide such a characterization. 
Finally, the main theorem that we shall prove in this section is not exactly \cref{thm:ExpressCycles}, 
but a variant of it (see \cref{thm:ExpressStablePerm}),
involving the notion of {\em stable subpermutation}. 

For $\pi \in \S_k$ and $\sigma \in \S_n$, we say that $\pi$ is a \emph{stable subpermutation} of $\sigma$ 
if there is some $k$-element subset $\Sigma \subseteq [n]$ such that $\sigma$ maps $\Sigma$ to $\Sigma$, and  the pattern of $\sigma$ on $\Sigma$ equals $\pi$. 
We call the set $\Sigma$ (or a sequence consisting of its elements) a \emph{stable occurrence} of $\pi$ in $\sigma$.  
That is, the stable subpermutations that $\sigma$ contains are just the subpermutations (or patterns) defined on unions of the cycles of $\sigma$. 

For example, $1$ is a stable subpermutation of $\sigma$ if and only if $\sigma$ has a fixed point, and the fixed points of $\sigma$ are precisely the stable occurrences of 1 in it. Or consider $\pi = 231$. The permutation $\sigma = 2413$ contains $\pi$ in the sense of pattern containment since the pattern of $241$ is $231$. But $\pi$ is not a stable subpermutation of $\sigma$ since $\sigma$ is a four-cycle so in fact its only stable permutations are the empty permutation and itself. On the other hand $\sigma = 4356712$ does contain $231$ as a stable subpermutation since $(1,4,6)$ is a cycle of $\sigma$ and the pattern of that cycle is $231$.

We now wish to consider the question: 
\emph{given $\pi \in \S_k$, in which permutation classes $\C$ is there a formula of $\phi(\mathbf{x})  \in \TOTOC$ with $k$ free variables such that for $\sigma \in \C$,  
$(\sigma,\mathbf{a}) \models \phi(\mathbf{x})$ if and only if $\mathbf{a}$ is a stable occurrence of $\pi$ in $\sigma$?}
To answer this question, we extend some ideas already presented in the fixed point case (in the proof of \cref{thm:fixed_point_expressible_in_classes} and in \cref{rk:fixed_point_expressible_in_classes}). 
The arguments being however more complicated, we first present an intermediate case: that of transpositions, i.e. stable occurrences of $21$.

\subsubsection{A formula recognizing transpositions}

About transpositions, we prove the following statement. 

\begin{theorem}
\label{thm:formula_for_transpositions}
  Let $\C$ be a permutation class. The following are equivalent.
  \begin{enumerate}
    \item There exists a formula $\phi(x,y) \in \TOTOC$ that expresses the fact that $(a,b)$ is a transposition in $\sigma$;
      namely, for $\sigma$ in $\C$ and $(a,b)$ a pair of elements of $\sigma$,
      $(\sigma,(a,b)) \models \phi(x,y)$ if and only if $(a,b)$ is a transposition of $\sigma$.
    \item There exist integers $k\geq 2$, $m\geq 1$ and $n\geq 1$ such that $\delta_k$ and $\inc_m \ominus \inc_n$
      do {\em not} belong to the class $\C$.
  \end{enumerate}
\end{theorem}

The strategy to prove \cref{thm:formula_for_transpositions} is the same as in the fixed point case. 
Given a transposition $(a,b)$ of a permutation $\sigma$, we identify subsets of elements of $\sigma$
whose cardinality must satisfy some constraints (like $|U| = |V|$ in the fixed point case). 
Then, as in \cref{rk:fixed_point_expressible_in_classes}, we determine necessary conditions on a permutation class $\C$ 
for \TOTOC to possibly express that a given pair of points is a transposition. 
And finally, we prove as in \cref{thm:fixed_point_expressible_in_classes} that these conditions are also sufficient. 

\medskip

Consider a permutation $\sigma$ where two points $a$ and $b$ forming an occurrence of $21$ have been identified. 
The designated copy of $21$ splits the remaining elements of $\sigma$ into nine regions forming a $3 \times 3$ grid, as in 
\[\sigma = \begin{array}{c}
	\begin{tikzpicture}[scale=0.25]
	\draw (0,0) grid (3,3);
	\fill (1,2) circle (0.15);
	\fill (2,1) circle (0.15);
	\end{tikzpicture}
\end{array}.\]
For $1 \le i,j \le 3$, we denote $a_{i,j}$ the number of elements in the region in
the $i$-th row and $j$-th column (indexing rows from bottom to top).
Call $A(\sigma,(a,b))=(a_{i,j})$ the corresponding $3\times 3$ matrix.
The designated points form a \emph{stable} occurrence of $21$ (i.e., a transposition) 
if and only if 
\begin{itemize}
  \item the number of other points of $\sigma$ in the first column equals those in the first row, 
    i.e. $a_{2,1}+a_{3,1}=a_{1,2}+a_{1,3}$ (we simplified a summand $a_{1,1}$, appearing on each side);
  \item and likewise for the second and third columns and rows, i.e.
    $a_{1,2}+a_{3,2}=a_{2,1}+a_{2,3}$ and $a_{1,3}+a_{2,3}=a_{3,1}+a_{3,2}$.
\end{itemize}
These equalities are the analogue of the characterization $|U|=|V|$
of fixed points, used in the previous section.
Note that the elements lying in any diagonal cell contribute equally to the corresponding sums,
and are therefore irrelevant.
We claim that a nonnegative matrix $A$ satisfies the above equality 
if and only if it is a nonnegative linear combination $\sum_{i=1}^8 c_i A^{(i)}$,
where (indexing matrix rows from bottom to top):

\scalebox{.9}{$\
A^{(1)}=\begin{pmatrix}
  0&0&0\\
  0&0&0\\
  1&0&0
\end{pmatrix},\
A^{(2)}=\begin{pmatrix}
  0&0&0\\
  0&1&0\\
  0&0&0
\end{pmatrix},\
A^{(3)}=\begin{pmatrix}
  0&0&1\\
  0&0&0\\
  0&0&0
\end{pmatrix},\
A^{(4)}=\begin{pmatrix}
  0&0&0\\
  1&0&0\\
  0&1&0
\end{pmatrix}
$} \\
\scalebox{.9}{$\
A^{(5)}=\begin{pmatrix}
  1&0&0\\
  0&0&0\\
  0&0&1
\end{pmatrix},\
A^{(6)}=\begin{pmatrix}
  0&1&0\\
  0&0&1\\
  0&0&0
\end{pmatrix},\
A^{(7)}=\begin{pmatrix}
  0&1&0\\
  1&0&0\\
  0&0&1
\end{pmatrix},\
A^{(8)}=\begin{pmatrix}
  1&0&0\\
  0&0&1\\
  0&1&0
\end{pmatrix}
$}\\
From the equalities $a_{2,1}+a_{3,1}=a_{1,2}+a_{1,3}$, $a_{1,2}+a_{3,2}=a_{2,1}+a_{2,3}$ and $a_{1,3}+a_{2,3}=a_{3,1}+a_{3,2}$, 
the existence of such a decomposition is not hard to prove ``greedily'', 
i.e. choosing, for increasing $i$, each $c_i$ as large as possible. 
Details are skipped here since it actually follows 
as a particular case of the coming \cref{Lemma:Matching marginals}.
\begin{proof}[Proof of \cref{thm:formula_for_transpositions} (1. implies 2.)]
Let $\C$ be a permutation class. 
For each $\sigma$ in $\C$ and pairs $(a,b)$ of elements of $\sigma$ forming an inversion,
we consider the corresponding matrix $A(\sigma,(a,b))$ and its above decomposition.

Assume that this yields arbitrarily large coefficients $c_4$.
This means that entries $a_{2,1}$ and $a_{1,2}$ of $A(\sigma,(a,b))$ can be
simultaneously both arbitrary large.
Put differently, the class contains permutations of the form
$2413[\tau_1,1,\tau_2,1]$ with arbitrary large permutations $\tau_1$ and $\tau_2$.
Then, using Erd\H{o}s-Szekeres theorem as in \cref{rk:fixed_point_expressible_in_classes}, $\C$ contains arbitrarily large permutations
of one of the following four types ($n$ and $m$ denoting the sizes of the monotone segments):
\begin{center}
	\begin{tikzpicture}[scale=0.5]
	\begin{scope}[shift={(0,0)}]
		\draw (0,0) grid (3,3);
		\fill (1,2) circle (0.15);
		\fill (2,1) circle (0.15);
		\draw (0.1, 1.1) -- (0.9,1.9);
		\draw (1.1, 0.1) -- (1.9,.9);
		\draw node at (1.5,-1) {$\alpha^4_{1,n,m}$};
	\end{scope}
	\begin{scope}[shift={(4,0)}]
		\draw (0,0) grid (3,3);
		\fill (1,2) circle (0.15);
		\fill (2,1) circle (0.15);
		\draw (0.1, 1.1) -- (0.9,1.9);
		\draw (1.9, 0.1) -- (1.1,.9);
		\draw node at (1.5,-1) {$\alpha^4_{2,n,m}$};
	\end{scope}
	\begin{scope}[shift={(8,0)}]
		\draw (0,0) grid (3,3);
		\fill (1,2) circle (0.15);
		\fill (2,1) circle (0.15);
		\draw (0.9, 1.1) -- (0.1,1.9);
		\draw (1.1, 0.1) -- (1.9,.9);
		\draw node at (1.5,-1) {$\alpha^4_{3,n,m}$};
	\end{scope}
	\begin{scope}[shift={(12,0)}]
		\draw (0,0) grid (3,3);
		\fill (1,2) circle (0.15);
		\fill (2,1) circle (0.15);
		\draw (0.9, 1.1) -- (0.1,1.9);
		\draw (1.9, 0.1) -- (1.1,.9);
		\draw node at (1.5,-1) {$\alpha^4_{4,n,m}$};
	\end{scope}
	\end{tikzpicture}.
\end{center}
If this is the case,  
an EF-game argument using \cref{Prop:monotone_equivalent,Prop:MagicLemma}
shows that transpositions are not definable in $\C$.
Indeed, we simply consider a large permutation $\sigma$ as above with both segments representing monotone sequences of the same size
and the permutation $\sigma'$ obtained by adding a point in one of this monotone sequence.
Then if $(a,b)$ and $(a',b')$ denote the elements corresponding to the marked elements (i.e., the black dots) in $\sigma$ and $\sigma'$ respectively,
Duplicator wins the EF-games in $k$ rounds on $(\sigma,(a,b))$ and $(\sigma',(a',b'))$,
but $(a,b)$ is a transposition in $\sigma$, while $(a',b')$ is not in $\sigma'$. 
Therefore, for 1. to hold, for each $i \leq 4$, the class $\C$ must avoid a permutation $\alpha^4_{i,n,m}$ for some $n$ and $m$. 
For $i=1$, this implies that there exists $n$, and $m$ such that $\inc_n \ominus \inc_m$ does not belong to $\C$. 

The same argument looking at the coefficient $c_5$
where the large permutations are taken to be decreasing
implies the existence of a $k$ such that $\dec_k$ does not belong to $\C$. 
\end{proof}

\begin{proof}[Proof of \cref{thm:formula_for_transpositions} (2. implies 1.)]
Suppose that there exists $k \geq 2$, $n\geq 1$ and $m\geq 1$ such $\dec_k$ 
and $\inc_n \ominus \inc_m$ do not belong to $\C$. 
Suppose w.l.o.g. that $m \geq n$. 
Let $\sigma \in \C$ and let $(a,b)$ be a pair of elements of $\sigma$ forming a transposition. 
We recall that the matrix $A(\sigma,(a,b))$ then writes as a linear combination
$\sum_{i=1}^8 c_i A^{(i)}$.
We will prove that we can bound the possible values of the coefficients $c_4, c_5, \dots, c_8$.

If $c_4 > (m-1)(k-1)$, we have $\min(a_{1,2},a_{2,1}) > (m-1)(k-1)$ and
 by the Erd\H{o}s-Szekeres theorem,
 the corresponding regions of $\sigma$ contain either $\delta_{k}$ or $\inc_m$.
 Since $\delta_{k}$ is not in \C, they should both contain $\inc_m$.
 This is however impossible since then \C would contain $\inc_n \ominus \inc_m$.
 We conclude that $c_4 \le (m-1)(k-1)$.
 The same argument shows that $\max(c_4,\dots,c_8) \le (m-1)(k-1)$.
 This does however not apply to the coefficients $c_1$, $c_2$ and $c_3$.

 We conclude that, when $(a,b)$ forms a transposition,
 the non-diagonal coefficients of $A(\sigma,(a,b))$ are all bounded by $2(m-1)(k-1)$.
 Recall that an inversion $(a,b)$ is a transposition if and only if
 \begin{equation}
   a_{2,1}+a_{3,1}=a_{1,2}+a_{1,3},\quad
         a_{1,2}+a_{3,2}=a_{2,1}+a_{2,3}, \quad a_{1,3}+a_{2,3}=a_{3,1}+a_{3,2}
   \label{eq:CharcTranspo}
 \end{equation}
 A formula expressing that $(a,b)$ is a transposition can therefore be obtained 
 as a conjunction of two formulas. The first one simply says that $(a,b)$ is an inversion.
 The second one is a big disjunction over lists $(a_{1,2},a_{1,3},a_{2,1},a_{2,3},a_{3,1},a_{3,2})$
 in $\{0,1,2, \dots ,2(m-1)(k-1)\}^6$ satisfying Eq.\eqref{eq:CharcTranspo} of the fact that
 \begin{itemize}
   \item there are exactly $a_{1,2}$ elements smaller than $a$ and $b$ in value order 
     (i.e. in the first row of $A(\sigma,(a,b))$)
     and between $a$ and $b$ in position order (i.e. in the second column of $A(\sigma,(a,b))$).
   \item and similar conditions involving $a_{1,3},a_{2,1},a_{2,3},a_{3,1}$ and $a_{3,2}$.\qedhere
 \end{itemize} 
 \end{proof}

This example illustrates that the distinction between ``formulas that recognize cycles/stable occurrences'' 
and ``sentences satisfied if a cycle/stable occurrence exists'' is a real one. 
Namely, in the class $\D$ of all monotone decreasing permutations there is no formula that recognizes transpositions, 
but it is easy to write a sentence that is satisfied if and only if a transposition exists -- specifically that the permutation contain at least two points.

\subsubsection{Extension to larger stable permutations}
We now generalize  \cref{thm:formula_for_transpositions} to subpermutations of larger size. 
Let a permutation $\pi \in \S_k$ be given. 
Our general goal is to characterize those classes $\C$ for which there is a formula with $k$ free variables 
expressing that $k$ given points of a permutation, $\sigma$, form a stable occurrence of $\pi$. 

We start by generalizing the partition of the elements of $\sigma$ in regions presented in the transposition case. 
Suppose that $\sigma$ is a permutation and $\mathbf{s}$ is a specific occurrence of $\pi$ in $\sigma$. 
Then $\mathbf{s}$ partitions $\sigma \setminus \mathbf{s}$ into $(k+1)^2$ regions -- or \emph{cells} --
where two elements of $\sigma \setminus \mathbf{s}$ belong to the same cell if 
their positional and value relationships to the elements of $\mathbf{s}$ are identical. 
These cells are naturally arranged in a $(k+1) \times (k+1)$ grid 
where the elements of two cells in the same row share the same value relationships with $\mathbf{s}$ 
while elements of two cells in the same column share the same positional relationships with $\mathbf{s}$. 
We associate to the pair $(\sigma, \mathbf{s})$ a matrix of non-negative integers $A(\sigma, \mathbf{s})$ 
whose entry in row $i$ and column $j$ is the number of elements of $\sigma \setminus \mathbf{s}$ belonging to the cell in row $i$ and column $j$. 
A specific example is shown in Figure \ref{Figure: Grid of an occurrence}. 
As above, we index matrix rows from bottom to top,
to maintain a geometric correspondence between these matrices and the diagrams of the corresponding permutations. 

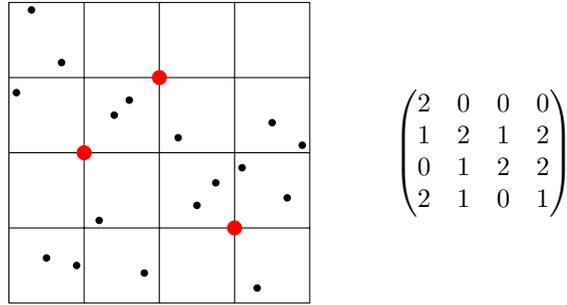
\begin{figure}
\begin{center}
	\begin{tikzpicture}[scale=1]
	\draw (0,0) grid (4,4);
	\fill[red] (1,2) circle (0.1);
	\fill[red] (2,3) circle (0.1);
	\fill[red] (3,1) circle (0.1);
	\fill (0.1, 2.8) circle (0.05);
	\fill (0.3, 3.9) circle (0.05);
	\fill (0.5, 0.6) circle (0.05);
	\fill (0.7, 3.2) circle (0.05);
	\fill (0.9, 0.5) circle (0.05);
	\fill (1.2, 1.1) circle (0.05);
	\fill (1.4, 2.5) circle (0.05);
	\fill (1.6, 2.7) circle (0.05);
	\fill (1.8, 0.4) circle (0.05);
	\fill (2.25, 2.2) circle (0.05);
	\fill (2.5, 1.3) circle (0.05);
	\fill (2.75, 1.6) circle (0.05);
	\fill (3.1, 1.8) circle (0.05);
	\fill (3.3, 0.2) circle (0.05);
	\fill (3.5, 2.4) circle (0.05);
	\fill (3.7, 1.4) circle (0.05);
	\fill (3.9, 2.1) circle (0.05);
	\draw (5,2) node[right] {$\left( \begin{matrix}
	2 & 0 & 0 & 0 \\ 1 & 2 & 1 & 2 \\ 0 & 1 & 2 & 2 \\ 2 & 1 & 0 & 1 
	\end{matrix} \right)$};
	\end{tikzpicture}
\end{center}
\caption{An occurrence of {\color{red} 231} in a permutation of $\S_{20}$ together with the corresponding matrix of cell sizes. The occurrence is not stable since corresponding row and column sums are not all equal.}
\label{Figure: Grid of an occurrence}
\end{figure}

An $m \times m$ matrix $A = (a_{ij})$ with non-negative integer entries will be said to have \emph{matching row and column marginals} if, for each $1 \leq i \leq m$,
\[
\sum_{j=1}^m a_{ij} = \sum_{j=1}^m a_{ji},
\]
i.e., the sum of entries in each row is equal to the sum of entries in the corresponding column. 
As in the case of transpositions, we have the following characterization of stable occurrences of $\pi$. 
\begin{observation}
\label{obs:matching_marginals}
An occurrence $\mathbf{s}$ of $\pi$ in a permutation $\sigma$ is stable if and only if the matrix $A(\sigma, \mathbf{s})$ has matching row and column marginals.
\end{observation}

Let $m$ be a positive integer, and let $\mathbf{i} = (i_1, i_2, \dots, i_r)$ be a sequence of distinct elements belonging to $[m]$. Let $A_\mathbf{i}$ be the $m \times m$ matrix whose entries in positions $(i_t, i_{t+1})$ for $1 \leq t < r$ and $(i_r, i_1)$ are 1 and whose other entries are 0. That is, $A_{\mathbf{i}}$ is the adjacency matrix of the directed cycle $i_1 \to i_2 \to \cdots \to i_r \to i_1$. We will call a matrix of this type a \emph{cycle matrix}. Note that if $r = 1$ then $A_{\mathbf{i}}$ contains a single element on the diagonal and we consider these to be \emph{trivial} cycle matrices (and the corresponding sequences $\mathbf{i}$ will also be called trivial).

\begin{lemma}
\label{Lemma:Matching marginals}
If $A = (a_{ij})$ is an $m \times m$ matrix with non-negative integer entries and matching row and column marginals, then it can be written as a linear combination with non-negative integer coefficients of cycle matrices.
\end{lemma}

\begin{proof}
Such a matrix $A$ is the adjacency graph (with multiplicities) of a directed multigraph (possibly including loops) on the vertex set $\{1,2,\dots,m\}$ where there are $a_{ij}$ directed edges from $i$ to $j$ (for each $1 \leq i, j \leq m$). The matching marginal condition implies that the indegree of each vertex is equal to its outdegree (loops contribute to both the indegree and outdegree of a vertex). It is clear that every such graph contains a directed cycle, and removing such a cycle leaves a graph of the same type. By induction, such graphs are unions of directed cycles, which yields to the claimed result.
\end{proof}

For $\pi \in \S_k$ and $\mathbf{i}$ a sequence of $r$ distinct elements from $[k+1]$ 
we define the \emph{expansion of $\pi$ by $\mathbf{i}$}, $E(\pi, \mathbf{i})$ to be 
that permutation of length $k + r$ containing a stable occurrence, $\mathbf{s}$ of $\pi$ 
and for which $A(E(\pi, \mathbf{i}), \mathbf{s}) = A_{\mathbf{i}}$. An example is given in Figure \ref{Figure: Expansion}. 
Given a sequence of $r$ permutations $\Theta = (\theta_1, \theta_2, \dots, \theta_r)$ we further define the \emph{inflation of $\pi$ by $\Theta$ on $\mathbf{i}$} 
to be the permutation, $E(\pi, \mathbf{i}, \Theta)$, obtained by inflating those points of $E(\pi, \mathbf{i})$ corresponding to the elements of $\mathbf{i}$ 
in left to right order by the permutations $\theta_1$, $\theta_2$, \dots, $\theta_r$. 
This naturally extends to inflation by permutation classes: 
if $\overline{\X} = (\X_1, \X_2, \dots, \X_r)$ is a sequence of permutation classes, 
then we define the \emph{inflation of $\pi$ by $\overline{\X}$ on $\mathbf{i}$} to be 
the permutation class, $E(\pi, \mathbf{i}, \overline{\X})$,  consisting of all the subpermutations of permutations in $E(\pi, \mathbf{i}, \Theta)$ 
where for $1 \leq t \leq r$, $\theta_t \in \X_t$.

\begin{figure}
\[
\revision{\begin{pmatrix}
  1&0&0&0&0\\
  0&0&0&0&0 \\
  0&0&0&0&0 \\
  0&0&0&0&1 \\
  0&1&0&0&0
\end{pmatrix}}
\qquad
\begin{array}{c}
\begin{tikzpicture}[scale=1]
	\draw (0,0) grid (5,5);
	\fill[red] (1,2) circle (0.1);
	\fill[red] (2,4) circle (0.1);
	\fill[red] (3,1) circle (0.1);
	\fill[red] (4,3) circle (0.1);
	\fill (1.5,.5) circle (0.1);
	\fill(4.5, 1.5) circle (0.1);
	\fill (.5,4.5) circle (0.1);
\end{tikzpicture}
\end{array}\]
\caption{\revision{Left: the matrix $A_{(1,2,5)}$.
Right: the diagram of the permutation $E(2413, (1,2,5)) = 7416253$,
enlightning the stable occurrence $\bm s$ of 2413 such that
$A(7416253,\bm s)=A_{(1,2,5)}$.}}
\label{Figure: Expansion}
\end{figure}
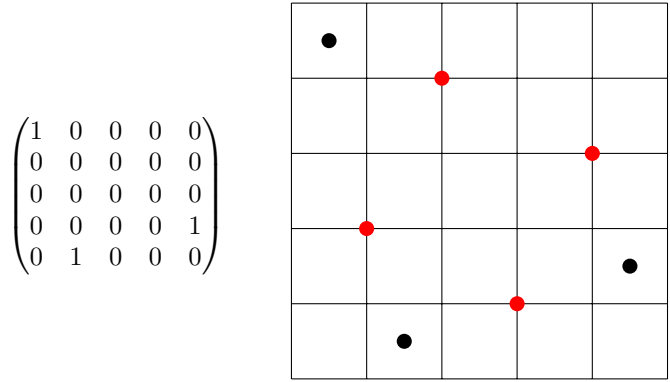

We can now state the following generalization of \cref{thm:formula_for_transpositions}. 
Note that the analogue of \cref{thm:formula_for_transpositions} is actually only the equivalence between $1.$ and $2.$, 
but it is proved by showing that $1. \Leftrightarrow 3. \Leftrightarrow 2.$, hence the third condition below.

\begin{theorem}
\label{thm:ExpressStablePerm}
Let $\pi$ be a permutation of size $k$, and let $\C$ be a permutation class. The following are equivalent. 
  \begin{enumerate}
    \item There exists a formula $\phi(x_1, \dots, x_k) \in \TOTOC$ that expresses the fact that 
    a sequence $\mathbf{s}=(s_1, \dots, s_k)$ is a stable occurrence of $\pi$ in $\sigma$.
    \item For each non-trivial cycle $\mathbf{i}$ of distinct elements from $[k+1]$ 
    and each sequence $\overline{\X}$ consisting of $\I$s and $\D$s of the same length 
    as $\mathbf{i}$ the class $\C$ does not contain $E(\pi, \mathbf{i}, \overline{\X})$, 
    i.e., it avoids at least one permutation in each such class.
    \item There is a positive integer $M$ such that if $\sigma \in \C$ and $\mathbf{s}$ is a stable occurrence of $\pi$ in $\sigma$ 
    then the sum of the non-diagonal entries of $A(\sigma, \mathbf{s})$ is at most $M$. 
  \end{enumerate}
\end{theorem}

As in the case of transposition, the key point of the proof is to bound the non-diagonal entries
in the matrices $A(\sigma, \mathbf{s})$, when $\mathbf{s}$ is a stable occurrence of $\pi$. 

\begin{proof}
That $3.$ implies $1.$ is easy and similar to the transposition case. 
Namely we build $\phi(x_1, \dots, x_k)$ as a conjunction of two formulas. 
The first one simply indicates that the chosen points of $\sigma$ form an occurrence of $\pi$. 
The second (and main) one will indicate that this occurrence is stable. 
It is obtained as follows. 
From the assumption in $3.$ and \cref{obs:matching_marginals}, this second part of $\phi(x_1, \dots, x_k)$ 
can be taken as a disjunction, over finitely many matrices, 
of formulas asserting that the number of entries in each (non-diagonal) cell relative to the purported occurrence of $\pi$ 
is given by the corresponding element of the matrix. 
The matrices that need to be considered are those of size $(k+1)\times(k+1)$ 
with matching row and column marginals, non-negative integer entries, zero entries on the diagonal, 
and such that the sum of all entries is at most $M$. 

Suppose now that $3.$ fails. 
For each positive integer $M$ choose a permutation $\sigma_M \in \C$ and a stable occurrence $\mathbf{s}_M$ of $\pi$ in $\sigma_M$ such that 
the sum of the non-diagonal entries of $A_M:=A(\sigma_M,\mathbf{s}_M)$ is greater than $M$.
(W.l.o.g., since $\C$ is a class, we assume that the diagonal entries of $A_M$ are all $0$.) 
For each $M$, $A_M$ satisfies the conditions of Lemma \ref{Lemma:Matching marginals} 
and so we can choose a representation of it as a linear combination with non-negative integer coefficients of matrices $A_{\mathbf{i}}$ 
for non-trivial cycles $\mathbf{i}$ of distinct elements from $[k+1]$ ($k+1$ being independent of $M$). 
Let $\mathbf{i}_M$ be that cycle for which the coefficient of $A_{\mathbf{i}}$ in the chosen representation is maximized. 
Note that, when $M$ grows to infinity, the coefficient of $A_{\mathbf{i}_M}$ also goes to infinity. 

The sequence $(\mathbf{i}_M)_M$ taking its values in a finite set, 
it contains an infinite subsequence whose elements are all equal to a single $\mathbf{i}$. 
Consider $E(\pi, \mathbf{i})$, the expansion of $\pi$ by $\mathbf{i}$. 
Since the coefficient of $A_{\mathbf{i}}$ is unbounded in the considered subsequence of $(A_M)$ we conclude that, 
for every $n$ there exists a permutation $\theta_n \in \C$ which is the inflation of $\pi$ on $\mathbf{i}$ by permutations of size at least $(n-1)^2 + 1$. 
Now take a monotone subsequence of length $n$ in each of these inflating permutations.
Passing to a subsequence again if necessary we can assume that for each $i_j$,
the type of the monotone subsequence by which we inflate $i_j$ does not depend on $n$.

In other words there is a sequence $\overline{\X}$ consisting of $\I$s and $\D$s such that $\C$ contains the class $E(\pi, \mathbf{i}, \overline{\X})$. 
But now we can apply Proposition \ref{Prop:MagicLemma} to conclude that there can be no formula of \TOTOC expressing stable occurrences of $\pi$, 
thus showing that $1.$ fails. 
Namely, for any $k$ we can now construct two $\EFeq_k$ equivalent permutations in $\C$ each with a marked occurrence of $\pi$ such that 
in one the occurrence is stable and in the other it is not. For the one where it is stable, 
we simply take monotone sequences all of the same sufficiently great length (e.g., length $2^k$) of the required types 
and form the inflation of $\pi$ on $\mathbf{i}$ by those sequences. 
For the other we take basically the same inflation but add a single point to any one of the sequences.

We are left with proving that $2.$ and $3.$ are equivalent. 
Obviously if $\C$ contains a class $E(\pi, \mathbf{i}, \overline{\X})$ 
then $3.$ fails. So suppose that $\C$ contains none of these (finitely many) classes.

This implies the existence of a positive integer $m$ such that for every $\mathbf{i}$
and every sequence $\Theta$ of the same length as $\mathbf{i}$ consisting of monotone (either increasing or decreasing) permutations of size $m$,
the permutation $E(\pi, \mathbf{i}, \Theta)$ does not belong to $\C$.
Let $C$ be the total number of non-trivial cycles on $[k+1]$ and take $M = C(k+1)((m-1)^2 + 1)$. 
We claim that if $\sigma \in \C$ contains a stable occurrence $\mathbf{s}$ of $\pi$ then 
the sum of the non-diagonal entries of $A(\sigma, \mathbf{s})$ is bounded above by $M$. 
Suppose this were not the case, and choose a counterexample $(\sigma, \mathbf{s})$. 
Using Lemma \ref{Lemma:Matching marginals} write $A(\sigma, \mathbf{s})$ as a non-negative linear combination of $A_{\mathbf{i}}$. 
In this decomposition of $A(\sigma, \mathbf{s})$, at most $C$ non-trivial $A_{\mathbf{i}}$ occur. 
Moreover, the number of (necessarily non-diagonal) entries of each such $A_{\mathbf{i}}$ is at most $k+1$. 
Since the sum of the non-diagonal entries of $A(\sigma, \mathbf{s})$ is greater than $M$,
this implies that there is some non-trivial cycle $\mathbf{i}$ such that the coefficient of $A_{\mathbf{i}}$ is at least $(m-1)^2 + 1$. 
In other words $\sigma$ contains a subpermutation which is an inflation of $\pi$ on $\mathbf{i}$ by permutations of size at least $(m-1)^2 + 1$. 
Again from the the Erd\H{o}s-Szekeres theorem, each of these permutations contains a monotone subsequence of length $m$, 
so $\sigma$ contains a permutation $E(\pi, \mathbf{i}, \Theta)$ where $\Theta$ is a sequence of monotone permutations each of length $m$ -- 
and that contradicts the choice of $m$.
\end{proof}

\subsubsection{From stable subpermutations to cycles}

Deducing the announced \cref{thm:ExpressCycles} from \cref{thm:ExpressStablePerm} is easy. 
Fix a permutation class \C and an integer $k$. 

Assume first that the second statement of \cref{thm:ExpressCycles} holds. 
This means that the second statement of \cref{thm:ExpressStablePerm} holds for any $k$-cycle $\pi$. 
For each such $\pi$, \cref{thm:ExpressStablePerm} ensures the existence of a formula 
$\phi_\pi(\mathbf{x})$ of \TOTOC with free variables $\mathbf{x} = (x_1, \dots, x_k)$ 
that expresses that $k$ distinguished elements of a permutation $\sigma$ form a stable occurrence of $\pi$. 
A formula $\phi(\mathbf{x})$ expressing that $k$ distinguished elements of a permutation $\sigma$ form a cycle 
is simply obtained as the disjunction of the $\phi_\pi(\mathbf{x})$ over all $k$-cycles $\pi$, 
proving that the first statement in \cref{thm:ExpressCycles} holds. 

Conversely, assume that the first statement \cref{thm:ExpressCycles} holds, 
that is to say that there exists a formula 
$\phi(\mathbf{x})$ of \TOTOC expressing that $k$ distinguished elements of a permutation $\sigma$ form a cycle. 
Fix a $k$-cycle $\pi$. It follows from Section~\ref{sec:patterns} that 
there exists a formula $\psi_\pi(\mathbf{x})$ expressing that 
$k$ distinguished elements of a permutation $\sigma$ form an occurrence of the pattern $\pi$. 
Therefore, that $k$ distinguished elements of a permutation $\sigma$ form a \emph{stable} occurrence of the pattern $\pi$ 
is simply expressed by $\psi_\pi(\mathbf{x}) \wedge \phi(\mathbf{x})$. 
From \cref{thm:ExpressStablePerm}, we deduce immediately the second statement of \cref{thm:ExpressCycles}. 

\medskip

We note that the above proof extends verbatim to yield the following statement. 

\begin{theorem}
Let $\C$ be a permutation class, and $k$ be an integer. The following are equivalent.
\begin{enumerate}
\item There is a formula $\phi(\mathbf{x})$ of \TOTOC with free variables $\mathbf{x} = (x_1, \dots, x_k)$ 
such that for all $\sigma \in \C$ and all sequences $\mathbf{a} = (a_1, \dots, a_k)$ of elements of $\sigma$, 
$(\sigma,\mathbf{a}) \models \phi(\mathbf{x})$ if and only if $\mathbf{a}$ is a union of cycles of $\sigma$.
\item For each permutation $\pi$ of size $k$, for each non-trivial cycle $\mathbf{i}$ of distinct elements from $[k+1]$, 
and for each sequence $\overline{\X}$ of the same length as $\mathbf{i}$, consisting of $\I$s and $\D$s, 
the class $\C$ avoids at least one permutation in each class of the form $E(\pi, \mathbf{i}, \overline{\X})$.
\end{enumerate}
\end{theorem}

\section{Intersection of \TOTO and \TOOB}
\label{Sect:Intersection}

%
\subsection{Results of this section}

The goal of this section is to characterize completely the properties of permutations which can be expressed both in \TOTO and in \TOOB. 
In other words, we want to identify the subsets of $\S = \cup_n \S_n$ for which 
there exist a sentence $\phi_{\TOTO}$ of \TOTO and a sentence  $\phi_{\TOOB}$ of \TOOB 
whose models are exactly the permutations in this set.

We start by introducing some terminology. 

\begin{definition}
\label{dfn:bi-fo}
A {\em \inter set} is a set $E$ of permutations such that 
there exist a sentence $\phi_{\TOTO}$ of \TOTO and a sentence  $\phi_{\TOOB}$ of \TOOB 
whose models are exactly the permutations in $E$.
\end{definition}

Recall that the  \emph{support} of a permutation is the set of its non-fixed points.
Our first result is the following. 
\begin{theorem}
Any \inter set $E$ either contains all permutations
with sufficiently large support,
or there is a bound on the size of the support of permutations in $E$.
\label{Thm:Intersection_True_Or_False}
\end{theorem}

Informally, this says that a property which can be described by a sentence 
in each of the two theories is, in some sense, trivial. 
Namely, it is either {\em eventually} true or {\em eventually} false,
where {\em eventually} means ``for all permutations with sufficiently large support''.
Note, however, that a property that can be expressed in both theories may be true for some, but not all, permutations with arbitrary large {\em size}. 
This is for instance demonstrated by the set $E=\{12\cdots n, n \ge 1\}$ (which have empty support), 
which is the set of models of the sentences 
$\phi_{\TOTO}= \forall x \, \forall y \ (x<_P y \leftrightarrow x <_V y)$
and $\phi_{\TOOB}=\forall x \, xRx$ 
of \TOTO and \TOOB respectively. 

There is a more precise version of the theorem,
which characterizes completely \inter sets $E$.
To state it, we introduce some notation.
For a partition $\la$, we denote by $\CCC_\la$
the set of permutations of cycle-type $\la$.
We also denote 
\begin{equation}
  \DDD_\la= \biguplus_{k \ge 0} \CCC_{\la \cup (1^k)}.
  \label{Eq:Def_DLa}
\end{equation}
\revision{(See the end of the introduction for the notation on 
integer partitions.)}
Finally, we recall that the \emph{Boolean algebra} generated 
by a family $\bm{\mathcal{F}}$ of subsets of $\Omega$
is the smallest collection of subsets of $\Omega$ containing $\bm{\mathcal{F}}$ and stable
by finite unions, finite intersections and taking complements.
Here, the role of $\Omega$ is played by the set $\SSS$ of all permutations (of all sizes).
\begin{theorem}
  A set $E$ of permutations is a \inter set
   if and only if
  it belongs to the Boolean algebra generated by all $\CCC_\la$ and $\DDD_\la$
  (where $\la$ runs over all partitions).
  \label{Thm:IntersectionBoolean}
\end{theorem}

\cref{Thm:Intersection_True_Or_False,Thm:IntersectionBoolean} are
proved in \cref{Sect:Proof_True_Or_False,Sect:Proof_Boolean}, respectively.

\begin{remark}
  We observe that \cref{Thm:IntersectionBoolean} is indeed more precise than \cref{Thm:Intersection_True_Or_False},
  in the sense that it is rather easy to deduce \cref{Thm:Intersection_True_Or_False} from \cref{Thm:IntersectionBoolean}.
  Indeed, for any $\la$, the sets $\CCC_\la$ and $\DDD_\la$ clearly have                                                          
  a bound, namely $|\la|$, on the size of the supports of permutations they contain.                                 
  Moreover, the property                                                                                                   
  \begin{quote}                                                                                                           
  ``Either $E$ contains all permutations                                                                                    
  with sufficiently large support,                                                                                          
  or there is a bound on the size of the support of permutations in $E$.''
  \end{quote}          
  is stable by taking finite unions, finite intersections and complements.
  Therefore, all sets in the Boolean algebra generated by the $\CCC_\la$'s and $\DDD_\la$'s
  satisfy this property and \cref{Thm:IntersectionBoolean} implies \cref{Thm:Intersection_True_Or_False},
  as claimed.
\end{remark}

\subsection{\inter sets are trivial}
\label{Sect:Proof_True_Or_False}

The goal of this section is to prove \cref{Thm:Intersection_True_Or_False}.
We consider a \inter set $E$. 
On one hand, $E$ is the set of models of
a sentence of \TOTO, whose quantifier depth is denoted $\ell$.
On the other hand, $E$ is the set of models of                           
a sentence in \TOOB. 
In particular, from \cref{prop:sameTOOB_conjugatePerm},
for any given partition $\la$,
it contains either all or none of the permutations of cycle-type $\la$.
To keep that in mind, everywhere in this section,
we write ``$E$ contains one/all permutation(s) of type $\la$''.

The general strategy of the proof is the following:
we identify a number of pairs of permutations $(\si_1,\si_2)$, 
on which Duplicator wins the $\ell$-move Duplicator-Spoiler game. 
We call such permutations $\si_1,\si_2$ \emph{$\ell$-indistinguishable} (or just \emph{indistinguishable}, $\ell$ being fixed throughout the proof).
For such a pair, $\si_1$ is in $E$ if and only if $\si_2$ is in $E$.
As a consequence, if $E$ contains one/all permutation(s) conjugate to $\si_1$,
it contains one/all permutation(s) conjugate to $\si_2$.

We can say even more:
\begin{lemma}
\label{cor:magic}
If there exist two permutations of cycle-types $\mu$ and $\nu$ that are indistinguishable, 
then for any fixed partition $\lambda$ 
there exists two permutations of cycle-types $\lambda \cup \mu$ and $\lambda \cup \nu$ that are indistinguishable.
\end{lemma}
\begin{proof}
  Let $\sigma$ be an arbitrarily chosen permutation of cycle-type $\la$
  and $\pi_1$ and $\pi_2$ be two permutations of cycle-types $\mu$ and $\nu$ that are indistinguishable.
  We consider $12[\sigma, \pi_1]$ and $12[\sigma, \pi_2]$. 
  From \cref{Prop:MagicLemma}, they are indistinguishable 
  and have cycle-type $\lambda \cup \mu$ and $\lambda \cup \nu$, respectively.
\end{proof}

In particular, under the assumption of the lemma,
our chosen \inter set $E$ contains one/all permutation(s) of cycle-type $\lambda \cup \mu$
if and only if it contains one/all permutation(s) of cycle-type $\lambda \cup \nu$.
\medskip

In the following few lemmas, we exhibit some pairs of indistinguishable
permutations with specific cycle-type, to which we will apply \cref{cor:magic}.
This serves as preparation for the proof of \cref{Thm:Intersection_True_Or_False}.

We denote $\ct(\si)$ the cycle-type of a permutation $\si$.

\begin{lemma}
  Let $n_1,n_2 \ge 2^\ell$.
  There exist indistinguishable permutations $\si_1$ and $\si_2$ 
  with $\ct(\si_1)=(n_1)$ and $\ct(\si_2)=(n_2)$.
  \label{Lem:N1_N2}
\end{lemma}

\begin{proof}
  Consider the permutations $\si_1=2\, 3\, \dots\, n_1\, 1$ 
  and $\si_2=2\, 3\, \dots\, n_2\, 1$. 
  They are clearly cyclic permutations of size $n_1$ and $n_2$, respectively.
  Moreover, they write as $\si_1=21[\inc_{n_1-1},1]$ and $\si_2=21[\inc_{n_2-1},1]$. 
  Of course, $\inc_{n_1-1}$ and $\inc_{n_2-1}$ are indistinguishable (see \cref{Prop:monotone_equivalent}), and then  
  \cref{Prop:MagicLemma} ensures that $\si_1$ and $\si_2$ are indistinguishable.
\end{proof}

With \cref{cor:magic}, a consequence of Lemma~\ref{Lem:N1_N2} for our \inter set $E$ is: 

\begin{corollary}
  Let $\la$ be a partition and $n_1,n_2 \ge 2^\ell$.
  Then $E$ contains one/all permutation(s) of cycle-type $\la \cup (n_1)$,
  if and only if it contains one/all permutation(s) of cycle-type $\la \cup (n_2)$.
  \label{Corol:N1_N2}
\end{corollary}

We now play the same game with other pairs of permutations/cycle-types.
\begin{lemma}
  Let $k$ and $m$ be non-negative integers with $(k-1)\, m \ge 2^\ell$.
  There exist indistinguishable permutations $\si_1$ and $\si_2$ 
  with $\ct(\si_1)=(km+1)$ and $\ct(\si_2)=(k^m)$.
  \label{Lem:KM}
\end{lemma}

\begin{proof}
  Consider the permutations $\si_1$ and $\si_2$ whose one-line representations are  
  \begin{align*}
    \si_1&= m\!+\!1\ \, m\!+\!2\ \, \cdots \ \, km\ \, km\!+\!1\ \, 1\ \, 2\ \, \cdots\ \, m;\\
    \si_2&= m\!+\!1\ \,m\!+\!2\ \, \cdots \ \, km\ \, 1\ \, 2\ \, \cdots\ \, m.
  \end{align*}
  Their decomposition into products of cycles of disjoint support are 
  \begin{align*}
    \si_1= \big(&1\ \,m\!+\!1\ \, 2m\!+\!1\ \, \cdots \ \, km\!+\!1\ \, m\ \, 2m\ \, \cdots\ \, km\ \, \\
    & m\!-\!1\ \, 2m\!-\!1\ \, \cdots \ \, km\!-\!1\ \, \cdots \ \, 2\ \,m\!+\!2\ \, \cdots \ \, (k-1)m\!+\!2 \big);\\
    \si_2= \big(&1\ \,m\!+\!1\ \, \cdots\ \, (k\!-\!1)m\!+\!1\big)\,
    \big(2\ \,m\!+\!2\ \, \cdots\ \, (k\!-\!1)m\!+\!2\big) \\
    & \cdots \ \big( m\ \,2m\ \, \cdots\ \, km\big),
  \end{align*}
  so that $\ct(\si_1)=(km+1)$ and $\ct(\si_2)=(k^m)$ as wanted.
  Moreover we can write $\si_1=21[\inc_{(k-1)m+1},\inc_m]$ and $\si_2=21[\inc_{(k-1)m},\inc_m]$,
  which proves using \cref{Prop:MagicLemma} that 
  $\si_1$ and $\si_2$ are indistinguishable.
\end{proof}

From Lemma~\ref{Lem:KM} and \cref{cor:magic}, we obtain: 

\begin{corollary}
  Let $\la$ be a partition and $k,m$ as above.
  Then $E$ contains one/all permutation(s) of cycle-type $\la \cup (km+1)$,
  if and only if it contains one/all permutation(s) of cycle-type $\la \cup (k^m)$.
  \label{Corol:KM}
\end{corollary}

\begin{lemma}
  Let $n \equiv 2\ (\!\!\!\!\mod 4)$ be an integer at least equal to $2(2^{\ell}+1)$ and let \hbox{$k \ge 1$}.
  There exist indistinguishable permutations $\si^k_1$ and $\si^k_2$ whose cycle-types are 
  $\ct(\si^k_1)=(n-2,k)$  and $\ct(\si^k_2)=(n+k-1)$.
  \label{Lem:LargeCycleEatsEverything}
\end{lemma}
In the above lemma, note that $(n-2,k)$ is not strictly speaking a partition, since it is not known how $k$ compares to $n-2$. 
Here and everywhere, we therefore identify a partition with all its rearrangements. 

\begin{proof}
  We first look at the case $k=1$. Set $h=n/2$, which is an odd integer. 
  Consider the permutations $\si^1_1$ of size $n-1$ and $\si^1_2$ of size $n$ given in one-line notation by 
  \begin{align*}
    \si^1_1&= h\!+\!1\ \,\,h\!+\!2\ \,\,\cdots\ \,\,n\!-\!1\ \,\,h\ \,\,2\ \,\,3\ \,\,\cdots\ \,\,h\!-\!1\ \,\,1 ; \\
    \si^1_2&= h\!+\!2\ \,\,h\!+\!3\ \,\,\cdots\ \,\,n\ \,\, h\!+\!1\ \,\,2\ \,\,3\ \,\,\cdots\ \,\,h\ \,\,1.
  \end{align*}
  Their diagrams are represented on \cref{fig:LargeCycleEatsEverything}.
  In cyclic notation, using that $h$ is odd in the case of $\si^1_2$, we get:
  \begin{align*}
    \si^1_1&= \big(h\,\big)\, \big(1\ \,\,h\!+\!1\ \,\,2\ \,\,h\!+\!2\ \,\,3\ \,\,\cdots\ \,\,h\!-\!1\ \,\,n\!-\!1\big) ; \\
    \si^1_2&= (1\ \,\,h\!+\!2\ \,\,3\ \,\,h\!+\!4\ \,\,5\ \,\,\cdots\ \,\,n\!-\!1\ \,\,h\ \,\,h\!+\!1\ \,\,2\ \,\,h\!+\!3\ \,\,4\ \,\,\cdots\ \,\,n\!-\!2\ \,\,h\!-\!1\ \,\,n\big),
  \end{align*}
  so that $\ct(\si^1_1)=(n-2,1)$ and $\ct(\si^1_2)=(n)$ as wanted.
  Moreover we can write $\si^1_1=4321[\inc_{h-1},1,\inc_{h-2},1]$, while 
  $\si^1_2=4321[\inc_{h-1},1,\inc_{h-1},1]$. 
  \cref{Prop:MagicLemma} proves that $\si^1_1$ and $\si^1_2$ are indistinguishable as soon as $h-2 \ge 2^\ell-1$
  or equivalently $n \ge 2(2^\ell+1)$.
  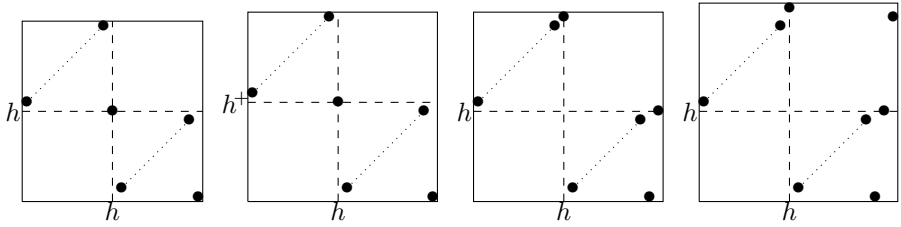
\begin{figure}[t]
    \[\hspace{-3mm} \begin{tikzpicture}[scale=.6]
      \draw (0,0)--(4,0)--(4,4)--(0,4)--(0,0);
      \draw[dashed] (0,2) -- (4,2);
      \draw[dashed] (2,0) -- (2,4);
      \absdot{(2,2)}{};
      \absdot{(3.9,0.1)}{};
      \absdot{(3.7,1.8)}{};
      \absdot{(2.2,.3)}{};
      \draw[dotted] (2.2,.3) -- (3.7,1.8);
      \absdot{(.1,2.2)}{};
      \absdot{(1.8,3.9)}{};
      \draw[dotted] (.1,2.2) -- (1.8,3.9);
      \node at (2,-0.2) {$h$};
      \node at (-.2,2) {$h$};
      
      \begin{scope}[xshift=5cm]
      \draw (0,0)--(4.2,0)--(4.2,4.2)--(0,4.2)--(0,0);
      \draw[dashed] (0,2.2) -- (4.2,2.2);
      \draw[dashed] (2,0) -- (2,4.2);
      \absdot{(2,2.2)}{};
      \absdot{(4.1,0.1)}{};
      \absdot{(3.9,2)}{};
      \absdot{(2.2,.3)}{};
      \draw[dotted] (2.2,.3) -- (3.9,2);
      \absdot{(.1,2.4)}{};
      \absdot{(1.8,4.1)}{};
      \draw[dotted] (.1,2.4) -- (1.8,4.1);
      \node at (2,-0.2) {$h$};
      \node at (-.25,2.2) {$h^{\!+}$};
      \end{scope}

      \begin{scope}
        [xshift=10cm]
      \draw (0,0)--(4.2,0)--(4.2,4.2)--(0,4.2)--(0,0);
      \draw[dashed] (0,2) -- (4.2,2);
      \draw[dashed] (2,0) -- (2,4.2);
      \absdot{(2,4.1)}{};
      \absdot{(4.1,2)}{};
      \absdot{(3.9,0.1)}{};
      \absdot{(3.7,1.8)}{};
      \absdot{(2.2,.3)}{};
      \draw[dotted] (2.2,.3) -- (3.7,1.8);
      \absdot{(.1,2.2)}{};
      \absdot{(1.8,3.9)}{};
      \draw[dotted] (.1,2.2) -- (1.8,3.9);
      \node at (2,-0.2) {$h$};
      \node at (-.2,2) {$h$};
      \end{scope}

      \begin{scope}
        [xshift=15cm]
      \draw (0,0)--(4.4,0)--(4.4,4.4)--(0,4.4)--(0,0);
      \draw[dashed] (0,2) -- (4.4,2);
      \draw[dashed] (2,0) -- (2,4.4);
      \absdot{(2,4.3)}{};
      \absdot{(4.1,2)}{};
      \absdot{(4.3,4.1)}{};
      \absdot{(3.9,0.1)}{};
      \absdot{(3.7,1.8)}{};
      \absdot{(2.2,.3)}{};
      \draw[dotted] (2.2,.3) -- (3.7,1.8);
      \absdot{(.1,2.2)}{};
      \absdot{(1.8,3.9)}{};
      \draw[dotted] (.1,2.2) -- (1.8,3.9);
      \node at (2,-0.2) {$h$};
      \node at (-.2,2) {$h$};
      \end{scope}
    \end{tikzpicture} \vspace{-6mm}\]
    \caption{From left to right: the permutations $\si^1_1$, $\si^1_2$,
    $\si^2_1:=\grow_h(\si^1_1)$ and \hbox{$\si^3_1:=\grow^2_h(\si^1_1)$} (we write $h^{\!+}:=h\!+\!1$).}
    \label{fig:LargeCycleEatsEverything}
  \end{figure}
  
  For larger values of $k$, we introduce an operator $\grow_h$ on permutations:
  by definition, for a permutation $\si$ of size $M$,
  the permutation $\grow_h(\si)$ has size $M+1$
  and is obtained from the word notation of $\si$ by replacing 
  $\si(h)$ with a new maximal element $M+1$
  and by appending the former value of $\si(h)$ to the right of $\si$.
  For an example, see \cref{fig:LargeCycleEatsEverything}.
  In cycle notation, this simply amounts to saying that $\grow_h(\si)$ is 
  obtained from $\si$ by inserting $M+1$ after $h$ in the cycle containing $h$.
  In particular, the size of the cycle containing $h$ is increased by $1$,
  while sizes of other cycles are left unchanged.

  We then set $\si^k_i=\grow_h^{\,k-1}(\si^1_i)$ (for $i=1,2$).
  Because of the above discussion on the effect of the operator $\grow_h$ on the cycle structure,
  it is clear that $\ct(\si^k_1)=(n-2,k)$ and $\ct(\si^k_2)=(n+k-1)$, as wanted.
  Moreover, if $k \ge2$, then $\si^k_1=462135[\inc_{h-1},1,\inc_{h-2},1,1,\inc_{k-2}]$, while 
  $\si^k_2=462135[\inc_{h-1},1,\inc_{h-1},1,1,\inc_{k-2}]$. 
  \cref{Prop:MagicLemma} ensures that they are indistinguishable as soon as $n \ge 2(2^\ell+1)$.
\end{proof}

\begin{corollary}
  Let $\la$ be a partition, $n,k$ be integers with $k \ge 1$ and $n \ge 2^\ell+2$.
  Then $E$ contains one/all permutation(s) of cycle-type $\la \cup (n-2,k)$,
  if and only if it contains one/all permutation(s) of cycle-type $\la \cup (n+k-1)$.
  \label{Corol:LargeCycleEatsEverything}
\end{corollary}

\begin{proof}
  For $n$ satisfying $n \equiv 2\ (\!\!\!\!\mod 4)$ and $n \ge 2(2^\ell+1)$, 
  the result follows from \cref{Lem:LargeCycleEatsEverything,cor:magic}. 
  We fix such a value $n_1$ of $n$.

  For a general $n \ge 2^\ell+2$, using the above case and twice \cref{Corol:N1_N2},
  we have:\\
  $E$ contains one/all permutation(s) of cycle-type $\la \cup (n-2,k)$\\
  if and only if it contains one/all permutation(s) of cycle-type $\la \cup (n_1-2,k)$,\\
  if and only if it contains one/all permutation(s) of cycle-type $\la \cup (n_1+k-1)$,\\
  if and only if it contains one/all permutation(s) of cycle-type $\la \cup (n+k-1)$.
\end{proof}

We can now proceed to the proof of
\revision{the first main result of this section}.

\begin{proof}
  [Proof of \cref{Thm:Intersection_True_Or_False}]
Let \[M= \sum_{k=2}^{2^\ell-1} k \cdot {\tfrac{2^\ell}{k-1}}.\] 
Throughout the proof, we fix a cyclic permutation $\zeta$ of size at least $2^\ell$. 
\medskip

Recall that for a partition $\la=(\la_1, \dots, \la_q)$, we denote by $|\la| = \sum_{i=1}^q \la_i$ its size. 
In addition, we denote by $\ell(\la)=q$ its number of parts. 

{\em Claim 1.}
Let $\tau$ be a permutation having a cycle of size at least $2^\ell$.
Then $E$ contains $\tau$ if and only if it contains $\zeta$.
\smallskip

{\em Proof of Claim 1.} 
By assumption, the cycle-type of $\tau$ is $(n_2) \cup \lambda$ for some integer $n_2 \ge 2^\ell$ and some partition $\lambda$. 
From Corollaries~\ref{Corol:N1_N2} (applied with an empty $\lambda$) and~\ref{Corol:LargeCycleEatsEverything}, it follows that 
\[\begin{tabular}{l} $E$ contains $\zeta$ \\
      $\Leftrightarrow$ $E$ contains one/all permutation(s) of cycle-type $(n_2+|\la|+\ell(\la))$\\
      $\Leftrightarrow$ $E$ contains one/all permutation(s) of cycle-type $(n_2+|\la|+\ell(\la)-\la_1-1,\la_1)$\\
      $\Leftrightarrow$ $E$ contains one/all permutation(s) of cycle-type $(n_2+|\la|+\ell(\la)-\la_1-\la_2-2,\la_1,\la_2)$\\
         \dots \\
      $\Leftrightarrow$ $E$ contains one/all permutation(s) of cycle-type $(n_2) \cup \la$\\
      $\Leftrightarrow$ $E$ contains $\tau$, 
    \end{tabular}\]
proving Claim~1.
\medskip

{\em Claim 2.}
Let $\tau$ be a permutation such that there exists $k$ 
between $2$ and $2^\ell-1$ such that $\tau$ has $m$ cycles of size $k$ with $(k-1) \, m \ge 2^\ell$.
Then $E$ contains $\tau$ if and only if it contains $\zeta$.
\smallskip

{\em Proof of Claim 2.} 
By assumption, the cycle-type of $\tau$ is $(k^{m}) \cup \la$ 
for some integers $k,m$ with $(k-1) \, m \ge 2^\ell$ and some partition $\la$.
From \cref{Corol:KM}, $\tau$ is in $E$ if and only if $E$ contains one/all permutation(s) of cycle-type $\la \cup (k\, m+1)$. 
Since $k\, m +1 \geq 2^\ell$, we deduce from our first claim that this happens exactly when $E$ contains $\zeta$,
proving Claim~2. 
\medskip

\revision{We come back to the proof of the theorem.}

\begin{itemize}
  \item 
Assume first that $E$ contains $\zeta$.
We show that $E$ contains all permutations whose support has size at least $M$. 
Let $\tau$ be any permutation whose support has size at least $M$. 

\revision{We consider first the case when} $\tau$ contains a cycle of size at least $2^\ell$. Since $E$ contains $\zeta$,
it follows from our first claim that $E$ contains $\tau$. 

\revision{We now consider the case when}
$\tau$ does not contain any cycle of size at least $2^\ell$. 
Denoting $m_k$ the number of cycles of size $k$ of $\tau$, 
since the support of $\tau$ is of size at least $M$, we have $\sum_{k=2}^{2^\ell-1}k\,m_k \geq M$. 
Because of the choice of $M$, this implies the existence of a $k$ such that $m_k \ge \tfrac{2^\ell}{k-1}$. 
And because $\zeta$ belongs to $E$, it follows from Claim 2 that $\tau$ belongs to $E$. 
  \item 
Assume now that $E$ does not contain $\zeta$.
From Claims 1 and 2, we know that for any permutation $\tau$ in $E$,
\begin{itemize}
  \item $\tau$ does not have a cycle of size $2^{\ell}$ or more;
  \item for each $k$ with $2\leq k \leq 2^\ell-1$,
    the number of cycles of size $k$ in $\tau$ is less than $\tfrac{2^\ell}{k-1}$.
\end{itemize}
Then the support of any permutation $\tau$ in $E$ is smaller than $M$,
concluding the proof of \cref{Thm:Intersection_True_Or_False}.
 \qedhere 
\end{itemize}
\end{proof}

\subsection{Characterization of \inter sets}
\label{Sect:Proof_Boolean}


%
%
%
%
%
The goal of this section is to prove \cref{Thm:IntersectionBoolean}.
Recall that $\CCC_\la$ (resp. $\DDD_\la$) denotes the set
of permutations that have cycle-type $\la$ (resp. $\la \cup (1^k)$ for some $k$).

We have seen in Section~\ref{sec:expressibilityC_laD_la} with Lemmas~\ref{prop:CLa_BiFO_TOTO} and~\ref{prop:DLa_BiFO_TOTO} 
that, for any partition $\la$, the property of belonging to $\CCC_\la$ (resp. $\DDD_\la$) is expressible in \TOTO. 
Moreover, $\CCC_\la$ is a conjugacy class, so being in  $\CCC_\la$ is clearly expressible in \TOOB. 
Finally, being in $\DDD_\la$ can be translated as follows:
  there are $|\la|$ distinct elements which forms a subpermutation of cycle-type $\la$,
  and all other elements are fixed points, i.e. satisfy $xRx$.
With this formulation, it is clear that being in $\DDD_\la$ is expressible in \TOOB. 

It follows that 
\begin{lemma}
  For any partition $\la$, $\CCC_\la$ and $\DDD_\la$ are \inter sets.
  \label{Lem:CLa_DLa_BiFO}
\end{lemma}

\begin{proof}[Proof of \cref{Thm:IntersectionBoolean}]
Call $\BA$ the Boolean algebra generated by
the $\CCC_\la$'s and the $\DDD_\la$'s. 
That the elements of $\BA$ are \inter sets is easy. 
It follows from \cref{Lem:CLa_DLa_BiFO}, since the family of \inter sets is closed by taking finite unions, finite intersections and complements

\medskip

Conversely, consider a \inter set $E$.
We want to prove that $E$ is in $\BA$.
By assumption,
there exists a \TOTO sentence,
whose quantifier depth we denote by $\ell$, of which $E$ is the set of models.

From \cref{Thm:Intersection_True_Or_False},
possibly replacing $E$ by its complement, we can assume that 
there is a bound $M$ on the size of the support of permutations in $E$.
In other words the cycle-type of a permutation in $E$ writes 
$\mu \cup (1^k)$ with $|\mu| \le M$ and $k \ge 0$.
From \cref{prop:sameTOOB_conjugatePerm}, two conjugate permutations are either both in $E$
or both outside $E$, so that $E$ is a (disjoint) union of conjugacy classes:
\begin{equation}
  E = \biguplus_{\la \in \Lambda} \CCC_\la, 
  \label{Eq:E_BigUplus}
\end{equation}
for some subset $\La$ of $\{\mu \cup (1^k), |\mu| \le M \text{ and }k \ge 0\}$.
Note that \cref{Eq:E_BigUplus} does not prove that $E$ is in $\BA$,
since the union may be infinite.

Observe that the decomposition $\la=\mu \cup (1^k)$ is unique if
we require that $\mu$ has no parts equal to $1$ (which we write as $m_1(\mu)=0$).
Thus \cref{Eq:E_BigUplus} can be rewritten as
\begin{equation}
  E= \biguplus_{\genfrac{}{}{0pt}{}{\mu \text{ s.t.}}{|\mu| \le M,\, m_1(\mu)=0}} E_\mu, 
\text{ where }
E_\mu:= \biguplus_{\genfrac{}{}{0pt}{}{k \ge 0}{\mu \cup (1^k) \in \Lambda}}  \CCC_{\mu \cup (1^k)}. 
\end{equation}
The set of partitions $\mu$ with $|\mu| \le M$ is finite,
so that it is enough to prove that each $E_\mu$ lies in $\BA$.
For this, it is enough that $I_\mu = \{k \ge 0 \text{ s.t. } \mu \cup (1^k) \in \Lambda\}$ is finite or co-finite 
(then giving that $E_\mu$ is a finite union of $\CCC_{\mu \cup (1^k)}$ 
or the set-difference of $\DDD_\mu$ with a finite union of $\CCC_{\mu \cup (1^k)}$). 

\medskip

We claim that if $k,k'\geq 2^\ell$, then $k \in I_\mu$ if and only if $k' \in I_\mu$. 
Proving this claim will conclude our proof of Theorem~\ref{Thm:IntersectionBoolean}. 
Recall that by definition, $k \in I_\mu$ if and only if
$E$ contains one/all permutation(s) of cycle-type $\mu \cup (1^k)$. 
We know (see \cref{Prop:monotone_equivalent}) that the identity permutations of size $k$ and $k'$,
whose cycle-types are  $(1^k)$ and $(1^{k'})$ respectively,
are $\ell$-indistinguishable. 
From \cref{cor:magic}, there exist $\ell$-indistinguishable permutations $\si$ and $\si'$
with cycle-type $\mu \cup (1^k)$ and $\mu \cup (1^{k'})$, respectively.
Then $\si$ is in $E$ if and only if $\si'$ is in $E$
proving that $k \in I_\mu$ if and only if $k' \in I_\mu$.
%
%
%
\end{proof}

\section*{Acknowledgements}
The authors are thankful to Marc Noy for insightful discussions on the topic.
In particular, VF's first contact with first-order logic 
and expressibility questions was through a talk given by Marc Noy
on logic and random graphs in the workshop
``Recent Trends in Algebraic and Geometric Combinatorics'' in Madrid 
in November '13.

MB is partially supported by the Swiss National Science Foundation, under grant number 200021-172536.

\bibliographystyle{plain}
\bibliography{fo}

MA: Department of Computer Science, University of Otago, Owheo Building, 133 Union Street East, Dunedin 9016, New Zealand.\\
malbert@cs.otago.ac.nz

MB, VF: Institut f\"ur Mathematik, Universit\"at Z\"urich, Winterthurerstrasse 190, 8057 Z\"urich, Switzerland.\\
$\{$mathilde.bouvel,valentin.feray$\}$@math.uzh.ch
\end{document}